\documentclass{aic}

         \usepackage{theoremref}
          \usepackage{enumerate}
            \usepackage{tabularx}
           \newcommand{\lab}[1]{\label{#1}}                
           \newcommand{\thlab}[1]{\thlabel{#1}} 
   
	\usepackage{mathrsfs,bbm}

        \newcommand{\bee}{\begin{equation}}
        \newcommand{\ee}{\end{equation}}
        \newcommand{\non}{\nonumber}
        \newcommand{\bean}{\begin{eqnarray*}}
        \newcommand{\eean}{\end{eqnarray*}}
        \newcommand\eqn[1]{(\ref{#1})}
       \newcommand{\bel}[1]{\bee\lab{#1}}
         \catcode`@=11 \@addtoreset{equation}{section}
        
        \renewcommand\tilde[1]{\widetilde{#1}}

        \newcommand\se{\subseteq}
        \newcommand\sm{\setminus}
        \newcommand{\Bad}{\ensuremath {B}}
       
        \newcommand{\badp}{\ensuremath {{\mathscr{Z}}_{\mathrm{y}}}}
        \newcommand{\bade}{\ensuremath {{\mathscr{Z}}_{\mathrm{p}}}}
        \newcommand{\Pid}[1]{\ensuremath{D_{#1}}}
        \def \Ae {A_{\mathrm{e}}} 
        \def \Ap {A_{\mathrm{p}}} 
    
        \newcommand{\ww}{(n-1)^{-1}}  
        \newcommand{\ua}{A}
        \newcommand{\ub}{B}
        \newcommand{\uh}{K}
        \newcommand{\uj}{J}
        \newcommand{\Pf}{P^{*}}
        \newcommand{\Rf}{R^{*}}
        \newcommand{\Rfs}{R^{*}}
        \newcommand{\Sf}{Y^{*}}
        \newcommand{\proj}{_{\circ \!}}
        \newcommand{\Lh}{\Lambda^{k}}
        \def \sp {Y}
          \newcommand{\Pft}{P^{*}}
        \newcommand{\Rft}{R^{*}}
        \newcommand{\Sft}{Y^{*}}
        \newcommand{\relf}[2]{\left( 1+ O \left( \frac{#1}{#2} \right) \right)}
           \newcommand{\rel}[1]{\left( 1+ O \left( #1 \right) \right)}
        \newcommand{\unn}{+}

        \def\eps{\varepsilon}
        \def \pseta{\psi_k} 
	\newcommand{\cH}{\mathcal{H}}
        \newcommand{\cT}{\mathcal{T}}
        \newcommand{\Hknm}{\mathcal{H}_k(n,m)}
        \newcommand{\Bknm}{\mathcal{B}_{k}(n, m)}
        \newcommand{\Dknm}{{\cal D}_k(n, m)}
        \def\Num{\mathcal{N}}
        
        \def\Pc{{\cal P}}
        \def\Rc{{\cal R}}
        \def\Sc{{\cal Y}}
        \def\Cc{{\cal C}}
        \def\cC{{\cal C}}

        \def \Vp {{\cal V}_{\mathrm{p}} }
        \def \Vs {{\cal V}_{\mathrm{y}} }
        
        \def\D{{\mathfrak{D}}}	
        \def\R{\mathbbm{R}}		
        \newcommand{\reals}{\R}
        \def\Z{\mathbbm{Z}_+}	
        \def\W{{\mathfrak{W}}} 	
        \def\pr{{\bf P}}			
        \def\Pr{\pr}			
        \def\Var{{\bf Var}}

        \def\ex{{\bf E}}
        \def\dv{{\bf d}}
        \def\sv{{\bf y}}
        \def\dw{{\bf d}'}
        \def\pv{{\bf p}}
        \def\rv{{\bf r}}
        \def\reg{\mathrm{reg}}
        
        \def\ve{{\bf e}}
        \def\eb{{\ve_b}}
        \def\ea{{\ve_a}}
        \def\ev{{\ve_v}}

        \def\dvB{\dv}

        \newcommand{\size}[1]{\ensuremath{\left| #1 \right|}}
        
        \newtheorem{thm}{Theorem}[section]
        \newtheorem{cor}[thm]{Corollary}
        \newtheorem{conj}[thm]{Conjecture}
        \newtheorem{lemma}[thm]{Lemma}
        \newtheorem{proposition}[thm]{Proposition}
        
        \newtheorem{claim}[thm]{Claim}
        
        \theoremstyle{definition}
        \newtheorem{remark}[thm]{Remark}

\aicAUTHORdetails{%
  title = {Asymptotic Enumeration of Hypergraphs by Degree Sequence}, 
  author = {Nina Kam\v cev, Anita Liebenau, and Nick Wormald},
  plaintextauthor = {Nina Kamcev, Anita Liebenau, Nick Wormald},
  keywords = {asymptotic enumeration, hypergraphs, degree sequence},
}   

\aicEDITORdetails{%
   year={2022},
   number={1},
   received={19 August 2020},   
   published={3 February 2022},  
   doi={10.19086/aic.32357},      
}   

\begin{document}

\begin{frontmatter}[classification=text]

\title{Asymptotic Enumeration of Hypergraphs by Degree Sequence\titlefootnote{Research supported by ARC Discovery Project DP180103684. }} 

\author[nina]{Nina Kam\v cev}
\author[anita]{Anita Liebenau}
\author[nick]{Nick Wormald}

\begin{abstract}
We prove an asymptotic formula for the number of $k$-uniform hypergraphs with a given degree sequence, for a wide range of parameters. In particular, we find a formula that is asymptotically equal to the number of $d$-regular $k$-uniform hypergraphs on $n$ vertices provided that $dn\le c\binom{n}{k}$ for a constant $c>0$, and {$3 \leq k < n^C$ for any $C<1/9.$} 
        Our results relate the degree sequence of a random $k$-uniform hypergraph to a simple model of nearly independent binomial random variables, thus extending the recent results for graphs due to the second and third author.
\end{abstract}
\end{frontmatter}

    \section{Introduction}

        Enumeration is a central topic in combinatorics. Classical results include, for example, Cayley's formula for the number of $n$-vertex trees~\cite{cayley89} and Moon's formula of the number of trees with a given degree sequence~\cite{moon70}. Enumeration of more complex graphs tends to be difficult, even if one asks only for approximate (asymptotic) formulae.
        Of particular interest has been the problem of finding formulae for the number of $d$-regular graphs on $n$ vertices. No exact formulae in closed form are known for this problem, not even when $d=3$. In 1959,  Read~\cite{read59} found a recursive relation and an asymptotic formula for the number of 3-regular graphs. This was extended by several researchers~\cite{bc78,lw17,mckay85,mw90,mw91,wormald78} and now an asymptotic formula for the number of $d$-regular graphs is known for all $ d \in [1, n-1]$. 
        In fact, these results provide asymptotic formulae for the number of graphs with a given degree sequence, not just regular, for a wide range of sequences. The focus of this work is on $k$-uniform hypergraphs, which have been well studied in the context of enumeration. The regular case is considered by Cooper, Frieze, Molloy and Reed~\cite{cfmr96} and by 
        Dudek, Frieze, Ruci\'nski and \v{S}ileikis~\cite{dfrs13} for the rather sparse range and for constant $k$, and by Kuperberg, Lovett and  Peled~\cite{klp17} for the dense range and sufficiently large $k$. Blinovsky and Greenhill~\cite{bg16,bg16lin} provided asymptotic formulae for irregular `sparse' degree sequences and allowed for $k$ to increase moderately with $n$. We provide more details on the range of these results further below.
        
        Hypergraph enumeration is closely related to properties of random or `typical' hypergraphs. In the case of graphs, for instance, McKay and Wormald~\cite{mw97} showed that if certain asymptotic fomulae for the number of graphs with a given degree sequence hold, then the degree sequence of the random graph $G(n,p)$ can be modelled by a sequence of independent binomial random variables. Furthermore, the above-mentioned hypergraph enumeration results were used in studying perfect matchings~\cite{cfmr96} and Hamilton cycles~\cite{agir16} in random regular hypergraphs, as well as enumerating linear hypergraphs with a given degree sequence~\cite{bg16lin}. 
       Dudek, Frieze, Ruci{\'n}ski and {\v{S}}ileikis~\cite{dfrs17} gave a result that `embeds' the Erd\H{o}s-R\'enyi hypergraph inside the random regular hypergraph, thereby extending Kim and Vu's well known result for graphs to a denser case and also to hypergraphs. As they point out in their introduction, a significant obstacle was the lack of suitable enumeration formulae (as used by Kim and Vu), and such formulae may lead to another proof of their hypergraph result. The present paper provides such formulae. In addition, enumeration results were used in~\cite{egm19}  for uniformly sampling  graphs with a given power law degree sequence using  Monte Carlo Markov chains (MCMC). Uniform sampling of hypergraphs using MCMC has not advanced as far theoretically as the graph case, and our enumeration results may have implications there as well. Finally, as pointed out in~\cite{xyh13}, enumeration of hypergraphs with given degrees is of interest in statistics of social and other networks. Xi, Yoshida and Haws~\cite{xyh13}  apply sequential importance sampling in estimating the number of multi-way 0-1 tables with given marginals, which correspond to hypergraphs with appropriate degree constraints.

        Our contribution includes asymptotic enumeration of $k$-uniform hypergraphs with a given degree sequence, for wide a range of parameters not covered previously. In particular, we show that there is a constant $c >0$ such that if {$k \leq n^C$ for some $C<1/9$} and $d < \frac ck \binom {n-1}{k-1}$, the number of $d$-regular $k$-uniform hypergraphs on $n$ vertices is asymptotically equal to
        	\bel{eq:regularcount} 
        	    \binom{ \binom nk }{dn/k} \binom{{n \binom{n-1}{k-1} }}{dn}^{-1}  
         \binom{ \binom{n-1}{k-1}}{d}^n
           e^{ (k-1)/{2}}.
            \ee

    When $d$ exceeds our bound and $k \geq C'$ for a sufficiently large constant $C'$, then the results of~\cite{klp17} apply. Combining this result with ours covers the range $d \leq \frac 12 \binom{n-1}{k-1}$, and hence (by complementation), the full range for $d$, provided that $k \geq C'$.
           
        Moreover, we show that the degree sequence of a random $k$-uniform hypergraph with $m$ edges can be modelled by a sequence of independent binomial random variables, conditioning on their sum being $km$. 
        
            These developments parallel the history of enumerating $d$-regular graphs. Namely, the switching methods give results for $d = o(\sqrt{n})$~\cite{mckay85,mw91} whereas analytic techniques are useful for dense graphs, i.e.~roughly when $d \in (n/ \log n, 1/2)$~\cite{mw90}. The method of Liebenau and Wormald~\cite{lw17} covers the intermediate range, with a significant overlap, and this is the method we build on in the present work.

            Let us introduce some notation necessary for discussing our results. Given an integer $k$, a \emph{k-uniform hypergraph} (or \emph{k-graph}) on a vertex set $V$ is a collection of $k$-element subsets of $V$, which are called the \emph{edges} of the hypergraph. The \emph{degree} of a vertex $v$ in a $k$-graph $H$ is the number of edges containing~$v$. 
            A hypergraph is \emph{d-regular} if each vertex has degree~$d$. 
            The \emph{degree sequence} of a $k$-graph $H$ on the vertex set $V=[n]$ is the ordered $n$-tuple $(\deg_{H}(v): v \in V)$. We actually study a more general question of estimating the number of $k$-uniform $n$-vertex hypergraphs with a given degree sequence $\dv = (d_1, \dots, d_n)$, as the number of vertices $n$ tends to infinity. All our asymptotic statements are with respect to $n$. The edge size $k$ and the given degree sequence $\dv$ may depend on $n$.

        We state our formulae in probabilistic language which allows us to apply established probabilistic techniques, as well as giving more  concise and intuitive formulae. Let $\Hknm$ be the uniform probability space on the set of $n$-vertex $k$-graphs with $m$ edges, and let $\Dknm$ denote the probability space induced by the degree sequence of $\Hknm$. The ground set of $\Dknm$ is $\{\dv \in \Z^n : \sum_{v \in [n]} d_v = km \}$, denoted by $\Omega$. 
        Note that, for a sequence $\dv \in\Omega$, the number of $k$-graphs on $n$ vertices with degree sequence $\dv$ is equal to 
        {$\pr_{\Dknm}(\dv)\left( {\binom nk \atop m }\right)$}, so enumeration of $k$-graphs with a given degree sequence is equivalent to estimating $\pr_{\Dknm}(\dv)$.  Our main result   relates $\Dknm$ to a much simpler space of independent binomial random variables. 
        Let $\Bknm$ be the probability space on $\Omega$ in which a sequence is chosen by sampling each component $d_v$ according to the binomial distribution $\text{Bin}\left(\binom{n-1}{k-1}, {p} \right)$, where $p \in (0,1)$,  conditioned on $\sum_v d_v = km$ and being otherwise independent.   We omit $p$ from the notation since each such $p$ defines the same probability space: for any vector $\dv\in\Omega$, we have 
        \bee \label{eq:condbinom}
            \pr_{\Bknm}(\dv) = \binom{n\binom{n-1}{k-1}}{km}^{-1}   \prod_{v \in [n]}\binom{ \binom{n-1}{k-1}}{d_v}.  
        \ee
        
 The following theorem establishes that the space $\Bknm$ can be used to model the degree sequence $\Dknm$, as long as $k$ is not too large. We use $o(1)$  and $\omega(1)$ to denote functions of $n$ tending to 0 and $\infty$, respectively, as $n\rightarrow \infty$. We say that the convergence in the error term $o()$ holds uniformly for all $\dv$ in some set if there is a fixed upper bound $g(n)$ on that error term, where $g(n)\to 0$ as $n\to\infty$, for all $\dv$ in the given set. A similar interpretation holds for uniform convergence in $O()$ bounds.  
    \begin{thm} \thlab{t:main}
    {There exists a real constant $c>0$ such that the following holds for any positive constant $C<1/9.$} Let  $k, n$ and $m$ be integers with $3 \leq k < n^{C}$ and  $ \sqrt{n}< m < \frac{c}{k}\binom nk$.  Then there exists a set $\W$ that has probability $1-O \left(n^{-\omega(1)}\right)$ in both $ \Bknm$ and $\Dknm$, such that uniformly for all $k, m$ as above and all $\dv \in \W$ 
            \bee    \label{eq:binomapp}
                \pr_{\Dknm}(\dv)  =  \pr_{\Bknm}(\dv) (1 + o(1)).
            \ee
            \end{thm}
       The theorem is restricted to $k\ge 3$ because  the fact that it holds when $k=2$ was shown in the earlier paper~\cite{lw17}, and our proofs here do not cover  that case. The asymptotic equivalence between $\Dknm$ and $\Bknm$ asserted by the theorem immediately opens up possibilities for proving properties of $\Dknm$, i.e.\ the   degree sequence of a random hypergraph with a given number of edges.
 This also has immediate implications for the \emph{binomial random hypergraph}, denoted by $\cH_k(n, p)$, in which each edge is included by sampling independently with probability $p$:  the distribution of the degree sequence of $\cH_k(n, p)$ is that of $\Dknm$ where $m$ has the distribution of $\text{Bin}\left(\binom{n}{k}, p \right)$. Precise  results of this sort were given in~\cite{mw97} for the case $k=2$, where the distribution of degrees in $\cH_2(n, p)$ is given a description reminiscent of de Finetti's theorem, as a mixture of finite sequences of independent binomial random variables. This approach is leading to results on the degree sequence of random graphs that are much stronger than obtained by the traditional approach, which often involves computing moments directly in the graph model.  By similar means, one can expect that some of the statistics of the degree sequences of $\cH_k(n, m)$ and $\cH_k(n, p)$ to be accessible from studies of independent copies of $\text{Bin}\left(\binom{n-1}{k-1}, p \right)$.

          {Our enumeration results which imply   \thref{t:main}  apply to a wider range of degree sequences than the set $\W$ referred to in \thref{t:main}, extending to degree sequences where the conditioned binomial model $\Bknm$   no longer gives asymptotically correct probabilities. These results are given in the two following theorems, whose proofs are related.}
            The following theorem implies \thref{t:main} in the moderately dense case, that is, when {the average degree} $d{=km/n}$ is at least polylogarithmic in $n$, and when the uniform edge size~$k$ is not too large with respect to $d$. The sequences in that theorem are nearly regular, by which we mean that all their components do not deviate significantly from the average $d$.  
            Define $\sigma^2(\dv) = \frac{1}{n}\sum_{v \in [n]}(d_v - d)^2$.    
            \begin{thm} \thlab{t:denseCase}
        			There exists $c>0$ such that the following {statements hold}. Let  $k$, $n$ and $ m$ be integers such that $k \geq 3$, let   $\varphi \in \left(\frac{4}{9}, \frac{1}{2} \right)$ and define $d = km/ n,$ $\mu = m \binom{n}{k}^{-1}$. Let $\D$ be the set of sequences $\dv$ of length $n$ satisfying $\sum_{v \in V} d_v = dn$ and $|d_v -d|\leq d^{1-\varphi}$ for all $v \in V$.  Suppose moreover that 
        			\bel{eq:conds}
        			    k^2 {\log^2 n} / \sqrt{n} {<c},\  k^3 {d^{1-3\varphi}} {<c}, \ k\mu <c \text{ and }   \log^Cn = o(d)  \text{ for any   fixed  integer } C.
        			   \ee		
        		Then uniformly  for all $\dv \in \D$ we have
                \bee    \label{eq:degseqformula}
                \pr_{\Dknm}(\dv ) = \pr_{\Bknm}(\dv )\exp \left( \frac{k-1 }{2} - \frac{(k-1) n \sigma^2(\dv )}{2d(n-k)(1-\mu)} \right) \rel{\eta_k} , 
                \ee
                where 
                $$\eta_k = 
                    \begin{cases}  \displaystyle 
                        \frac{ {\log^2 n} }{\sqrt{n}} + \frac{d^{2-4\varphi}}{n} + d^{1 - 3 \varphi} , & k=3
                        \\
                    \displaystyle    \frac{k^2 {\log^2 n} }{\sqrt{n}} + (\mu n +k)k^2 d^{1-3 \varphi}, \rule{0cm}{7mm} & k \geq 4.
                    \end{cases} $$
        		{Moreover,} $\eta_k < 3c$ {for $n$ sufficiently large.}
            \end{thm}
                        
 For the 
           constant degree sequence $\dv_{\reg}=(d, d, \dots, d)$,  \thref{t:denseCase} implies~\eqn{eq:regularcount},  
and hence ${\dv_\reg}$ fails to satisfy~\eqn{eq:binomapp} (for $k\ge 2$). This can be seen in~\eqn{eq:degseqformula}:    $\dv_\reg$ occurs in $\Dknm$ with a higher probability than in $\Bknm$ because $\sigma^2(\dv_{\reg})=0$, which deviates significantly from the typical value of $\sigma^2(\dv)$ in $\Bknm$ and $\Dknm.$ 
           Still, Theorem~\ref{t:main} is deduced from \thref{t:denseCase}, combined with \thref{t:sparseCase} below,
            by considering the `annulus' consisting of degree sequences $\dv$ for which $\sigma^2(\dv)$ is \emph{close} to its expectation so that the exponential term in~\eqref{eq:degseqformula} tends to $1$.
                
            Our next result, like those of~\cite{bg16,bg16lin,cfmr96,dfrs13}, covers the \emph{sparse} regime, so we have no assumptions on the variation of the entries of the sequence and the errors are expressed in terms of maximum degree $\Delta(\dv)= \max_{v \in [n]} d_v$. The theorem is particularly important for complementing~\thref{t:denseCase} when $k$ is large compared to the average degree $d$. 
                         \begin{thm} \thlab{t:sparseCase}
                Let $n$, $k$, $\Delta_*$ and $m$ be {positive} integers.  Define
                \begin{eqnarray}
                     \psi_k = \begin{cases} \displaystyle 
                         \frac {1}{m} \left( \Delta_* ^{3} + \log^9 n \right) , & k=3, \\
           \displaystyle  \frac {k^2}{m} \left( \Delta_* ^2 + k \Delta_* \log^4 n +   k^2 \log^9 n \right), \rule{0cm}{7mm} & k \geq 4
        			    \end{cases}
        			\label{eq:sparse4}
                \end{eqnarray}
                and assume $\psi_k <1$.  
                {Uniformly for all  sequences} $\dv \in \Z^n$
                with  $\sum_{v \in [n]}d_v= km$ and $\Delta(\dv) \leq \Delta_*$, we have  
                $$\pr_{\Dknm}(\dv) = \pr_{\Bknm}(\dv)\exp \left( \frac{(k-1)}{2}\left(1- \frac{ n  \sigma^2(\dv)}{km} \right)\right) \rel{\sqrt{\psi_k}}  .$$
                \end{thm}

            As a  byproduct  of our method, we estimate the probability that a given set $K \in \binom Vk$ is a hyperedge in a random $k$-graph with degree sequnce $\dv$, denoted by $P_K(\dv)$. As expected, the approximate probability is `close to' the edge density $\mu$ and increases with the degrees of the vertices in $K$. 
            \begin{proposition} \thlab{Pdense}
        			There exists a real constant $c>0$ such that the following holds. Let $k\geq3$, $d$, $n$, $m$, $\mu$, $\varphi$ and $\D$ be as in \thref{t:denseCase}, and replace the assumptions~\eqref{eq:conds} by
        			$$  \ {kd^{-\varphi}<c },\  k \mu + k^2/n < c \text{ and } \log^3 n <d.  $$ 	
        		Then uniformly for all $\dv \in \D$ and $K \in \binom{V}{k}$, we have
                \bee    
                    P_K(\dv) = \binom{n-1}{k-1}^{-1} \left ( d + \frac{n-1}{n-k}\sum_{v \in K}\left(d_v - d \right) \right)\rel{k^2d^{-2\varphi}}.
                \ee
            \end{proposition}
            
            The proposition follows from the proof of \thref{l:ratiosDense} which allows slightly weaker assumptions on $d$ and $k$ than~\eqref{eq:conds}. We include the proof in Section~\ref{sec:densecase}. Analogously, we establish a more precise approximation for $k=3$, cf.~\eqref{eq:rf3} and~\eqref{eq:PSproof}, as well as for the edge probability $P_K(\dv)$ in the sparse range of~\thref{t:sparseCase}, cf.~\eqref{eq:Psparse}.  Finally, approximations for the probability that a random $k$-graph with degree sequence $\dv$ contains a specific two-edge path (consisting of two edges sharing $k-1$ vertices) can be deduced from~\eqref{eq:PSproof}. 
            
            \subsubsection*{Existing results and methods}
             We now give   a more detailed context for our results. The strongest theorem in the sparse regime is due to Greenhill and Blinovsky~\cite{bg16lin}. They enumerate $k$-graphs with degree sequence $\dv$ satisfying $k^3  \Delta(\dv)^3 = o(m)$. This condition implies $k^2 \Delta(\dv)^2 + k^3\Delta(\dv) = o(m)$, which is the bottleneck in~\thref{t:sparseCase} assuming a mild lower bound on $m$. The methods of~\cite{bg16,bg16lin,cfmr96,dfrs13} are based on two key ideas  which were initially developed for graphs --- the configuration model combined with switchings. 
             
             On the other hand, the enumeration formula of Kuperberg, Lovett and Peled~\cite{klp17} covers $d$-regular $k$-graphs with $kn^C \leq d \leq \frac 12 \binom {n-1}{k-1}$, where $C>0$ is an unspecified constant. In the range where both our results and those of~\cite{klp17} are valid, we have checked that~\eqref{eq:regularcount} asymptotically coincides with their formula. %
         Regular hypergraphs are actually a special case of a more general result from~\cite{klp17} showing the existence of various regular structures, such as designs and orthogonal arrays. The method used bears a strong relationship with the saddlepoint integral method first applied to enumeration of dense regular graphs in~\cite{mw90} and further developed in a number of papers by McKay and others.

       Finally, the reader may wonder whether enumeration results of bipartite graphs can be of use since a $k$-uniform hypergraph $H$ naturally corresponds to a bipartite graph $B_H$ on the vertex set $V_H \cup E_H$ in which each vertex of $E_H$ has degree $k$. 
      However, in this bipartite graph,         no two elements of   $E_H$ have the same set of neighbours  (as $H$ does not contain a multiple edge). Furthermore, $|E_H|$ can be as large as $\binom{|V_H|}{k}$, which means that the bipartite graphs have very disparate degrees. The existing enumeration results on bipartite graphs do not cover either of those aspects.

          \subsection*{Our method}
        The proof of Theorem~\ref{t:denseCase} follows the approach developed by Liebenau and Wormald~\cite{lw17}.  We will now informally describe several key insights. Let $\pr={\pr }_{\Dknm}$ denote the probability measure in $\Dknm$.  Firstly, it is sufficient to know the ratios $\pr (\dv) / \pr (\dw)$ for $\dv, \dw$ in the domain of interest $\D$ (which contains \emph{nearly regular} sequences with sum $nd$). However, to see that  it suffices to restrict our attention to $\D$, one has to check that $\D$ comprises most of the probability space $\Dknm$ and that the ratios $\pr(\dv)/ \pr(\dv')$  are \emph{typically} very close to the corresponding ratio in $\Bknm$, for which concentration results are crucial.
        
        Let us call sequences $\dv$ and $\dv'$ \emph{adjacent} if $\dv' = \dv - \ea + \eb$ for some canonical-basis vectors $\ea$ and $\eb$. Our argument relies on computing the ratios only for pairs of adjacent sequences and comparing them to the ratios in $\Pr_{\Bknm}$. Namely, for $\dv-\ea$ and $\dv-\eb$ in $\D$,  we have
             $$
             \frac{\pr_{\Bknm}(\dv-\ea)}{\pr_{\Bknm} (\dv- \eb)} = \frac{d_a \left(\binom{n-1}{k-1}-d_b+1 \right) }{d_b \left(\binom{n-1}{k-1}-d_a+1 \right) }
             $$
             and we will establish  (with a `very small' error $\eps$) that 
            \bee
                \frac{{\pr }_{\Dknm}(\dv- \ea)}{{\pr }_{\Dknm}(\dv- \eb)} =  \frac{d_a \left(\binom{n-1}{k-1}- d_b +1\right)}{d_b \left( \binom{n-1}{k-1} -d_a +1 \right)}\exp\left( \frac{(k-1) (d_a-d_b)}{d(1-\mu)(n-k)} \right)  \rel{\eps}.
            \ee 
            The final formula~\eqref{eq:degseqformula} is deduced by comparing the two approximations.
        
        	 Finally, to compute the ratios between adjacent degree sequences, we develop combinatorial relations which, when iterated, allow us to estimate the ratios with increasing accuracy.  Still, to reach a sufficiently high accuracy, one would need roughly $\log (nd)$ iterations of the recursive relations, so fixed point arguments are used instead to `short-circuit' the iterations. The above-mentioned edge probabilities $P_K(\dv)$ are approximated in this part of the argument, as they are closely related to the ratios $\pr(\dv - \ea) / \pr(\dv - \eb)$.
        	 
  Compared to the graph case~\cite{lw17}, the combinatorial relations become more complicated when $k > 2,$ requiring some care to resolve. Furthermore, iterating the relations and identifying the negligible terms in the iterations is different when $k$ grows (quite quickly).

\subsection*{A conjecture}
{The   comparison of $\Dknm$ with the space  $\Bknm$ of conditioned independent binomials in Theorem~\ref{t:main}} coincides with the above-mentioned work for {the graph case, i.e.~when} $k=2$,  however the distribution of the degree of a single vertex of $\cH_k(n, m)$ is    hypergeometric with parameters ${n\choose k},m, {n-1 \choose k-1}$,
since $m$ edges are chosen uniformly at random out of $\binom nk$ $k$-element subsets of $[n]$, and $\binom{n-1}{k-1}$ of those sets contribute to $d_v$. 
This suggests using the probability space which is a sequence of $n$ variables with this {hypergeometric} distribution, conditioned on their sum being $km$. We denote this by $\cT_k(n, m)$. We use {the conditioned binomial probability space} in the present work because the formula \eqn{eq:condbinom} is simpler than the corresponding one for  $\cT_k(n, m)$, and {these two spaces} are asymptotically equal for typical sequences for $k$ up to about $\sqrt n$. {However, for larger values of $k$, the asymptotic equivalence of $\Bknm$  with $\cT_k(n, m)$ no longer holds,} and we conjecture that the binomial model property already shown for $k=2$ in~\cite{lw17} extends in the following way to all $k$.

  \begin{conj} \thlab{c:main}
   Let $k, n$ and $m$ be integers with $2 \leq k \le n-2$ and $\min\{  m, \binom nk-m\}=\omega(1) \log n  $.   Then there exists a set $\W$ that has probability $1-O \left(n^{-\omega(1)}\right)$ in both $\cT_k(n, m)$ and $\Dknm$, such that uniformly for all $\dv \in \W$ 
$$  
                \pr_{\Dknm}(\dv)  =  \pr_{\cT_k(n, m)}(\dv) (1 + o(1)).
$$
            \end{conj}

            \subsection*{Structure of the paper}  Section~\ref{sec:preliminaries} contains some prerequisite results. In Section~\ref{sec:recursions} we establish recursive relations which are used in both proofs. {Theorem~\ref{t:sparseCase} is proved in Section~\ref{sec:sparsecase}, and Theorem~\ref{t:denseCase} in Sections~\ref{sec:operators} and~\ref{sec:densecase}. At the end of Section~\ref{sec:densecase}  we deduce Theorem~\ref{t:main}.}

          \section{Preliminaries}      \label{sec:preliminaries}
            \subsection{Notation}
            {For later reference, we list here some of the symbols defined throughout the paper, with their interpretations.}\\ 

            \renewcommand{\arraystretch}{1.3}
            \begin{tabularx}{\textwidth}{ll}
                      \hline
                      {Symbol} & {\centering Meaning or value} \\
                      \hline 
                      $V=[n]$ & the set $\{1, 2, \dots, n \}$\\
                      $\Hknm$ & random $k$-uniform hypergraph on $m$ edges \\
                      $\Dknm$ & degree sequence of $\Hknm$\\
                      $\Bknm$ & sequence of $n$ independent binomials conditioned on $\sum d_v = km$ \\
                      $d$ & $mn / k$ \\
                      $\mu$ & $m \binom nk ^{-1}$\\
                      $q$ & $k-1$ \\
                      $\alpha$ & $\frac{k-1}{n-1} = \frac{q}{n-1}$ \\
                      $\varphi \in \left(\frac 49, \frac 12 \right)$ & bound on the degree spread, so that $|d_v-d| < d^{1-\varphi}$ \\
                      $\eps_v$ & relative degree {deviation} $(d_v-d)/d$ \\
                      $\Delta(\dv)$ & $\max_v d_v$ \\
                      $M_1(\dv)$ & $\sum_v d_v$ \\
                      $\sigma^2(\dv)$ & $(\sum_v (d_v-d)^2)/n$ \\
                     $\psi_3$ & $\frac {1}{m} \left( \Delta_* ^{3} + \log^9 n \right)$\\
                     $\psi_k$ for $k \geq 4$ & $  \frac {k^2}{m} \left( \Delta_* ^2 + k \Delta_* \log^4 n +   k^2 \log^9 n \right)$\\
                     $ \eta_3 $ & $
                        \frac{ {\log^2 n} }{\sqrt{n}} + \frac{d^{2-4\varphi}}{n} + d^{1 - 3 \varphi} $\\
                        $\eta_k$ for $k\geq 4$ & $\frac{k^2 {\log^2 n} }{\sqrt{n}} + (\mu n +k)k^2 d^{1-3 \varphi}, \rule{0cm}{7mm}$   \\

                      \hline
            \end{tabularx}
                \renewcommand{\arraystretch}{1}

                Let $\Num(\dv; k, n)$ be the number of $k$-graphs on vertex set $[n]$ with a given degree sequence $\dv \in \Z^n$. Recall that we are interested in the asymptotic behaviour of $\Num(\dv;k,n)$ as $n \rightarrow \infty$. The parameters $k$, $m$ and $d$ are functions of $n$ and denote the edge size, number of edges and average degree of $k$-graphs under consideration, respectively. It is occasionally convenient to use $q = k-1$ instead of $k$.  
                In some of our notation, the dependence on $k$ and $n$ are suppressed, for instance we write $\Num(\dv) = \Num(\dv; k, n)$.
                Recall the definitions of the probability spaces $\Hknm, \Dknm$ and $\Bknm$ from the introduction.

                Given $\dv \in \Z^n$ with average $d = km/n$, we call $\max_{v \in [n]}{|d_v - d| }$ the \emph{spread} of $\dv$. Recall that $\sigma^2 (\dv) = \frac 1n \sum_{v=1}^n(d_v-d)^2$. 
                        
                Let $\ve_\uh = \sum_{a \in K} \ve_a$ for an arbitrary set $K \subset V$, where $\ea$ is the canonical basis vector in direction $a$. Usually, $\uh$ will be a $k$-element set, so that $\dv - \ve_{\uh}$ is the sequence obtained by perturbing $\dv$ in $k$ components.
                 Moreover, we define the norm $\Lh(\dv) = \max \left \{ L^{\infty}(\dv), \frac 1k L^{1}(\dv) \right \}$ on $\Z^n$ and the corresponding $\Lh$-distance. This metric is useful in our problem since for $K \subset V$ with $|K|  = k$, we have $\Lh(\dv, \dv - \ve_{K}) = \Lh(\ve_K, {\bf 0}) =1$.%
                 
                  Recall that our asymptotic statements are with respect to $n \to \infty$. We write    $f\sim g$ if $f/g\to 1$,    $f=O(g)$ if $|f|\le Cg$ for some constant $C$,  and $f=o(g)$ if $f/g\to 0$.  We use $\omega(1)$ to mean a function going to infinity, possibly different in all instances.                    For a given set $S$ and a positive integer $t$, we define $\binom{S}{t}$  to be the set of all $t$-element subsets of $S$.

                    To prove that an asymptotic relation holds uniformly  when $\dv_n$ is in some well defined finite set $S(n)$ for all $n$, it is enough if we can conclude that the relation is valid for any particular   sequence $\{\dv_n\}$ with $\dv_n\in S(n)$ for all $n$. This is because we may then apply this conclusion to the sequence of $\dv_n\in S(n)$   chosen for each $n$ with the `worst' error term over all possible $\dv \in S(n)$. In this paper, when we prove that an estimate holds uniformly, one can either verify at each step that the estimates are uniform, or take the viewpoint that a fixed sequence $\{\dv_n\}$ is under examination. The claims are sometimes easier to verify from the latter viewpoint.

             \subsection{Combinatorial preliminaries}
                A sequence $\dv = (d_v)_{v \in [n]}$ is \emph{k-graphical} if there is a $k$-graph with degree sequence $\dv$. Corollary 2.2 from~\cite{bef} is sufficient for our purposes. 
                \begin{thm} \thlab{thm:bef}
                Let $\dv \in \Z^n$ be a sequence with maximum component $\Delta$ such that $\sum_{v \in [n]} d_v = dn \equiv 0 \pmod{k}$. If there is an integer $p$ satisfying $\Delta \leq \binom{p-1}{k-1}$ and $nd \geq (\Delta-1)p+1$, then~$\dv$ is $k$-graphical.
                \end{thm} %
        The following proposition contains two immediate consequences for nearly regular hypergraphs and for sparse hypergraphs.
                \begin{proposition} \label{prop:kgraphical}  
        			Let $n$ and $k$ be integers such that $20k<n$. 
        			Let $\dv \in \Z^n$ be a sequence with $\sum_{v \in [n]} d_v = nd \equiv 0 \pmod{k}$.
                    \begin{enumerate}[(i)]
                        \item If  $\Delta (\dv) \leq \left(1 + \frac{1}{3k} \right)  d \leq \frac 14 \binom {n-1}{k-1}$  and $d \geq 1$, then $\dv$ is $k$-graphical.
                        \item If 
        	${ k \Delta(\dv)^{1 + 1/(k-1)}} \leq  \frac 12 dn$,
        			then the sequence $\dv$ is $k$-graphical.
                    \end{enumerate}
                \end{proposition}
                \begin{proof}
                    \begin{enumerate}[(i)] 
                        \item Let $p = \left \lceil \left( 1 - \frac{1}{2k} \right)n \right \rceil $. Then
        			$$p(\Delta(\dv)-1)+1  < nd \left(1 - \frac{1}{2k} + \frac 1n \right)\left(  1+ \frac {1}{3k} \right) +1 < nd \left(1 - \frac{1}{6k} + \frac{3}{n} \right) <nd,$$
        			using the assumption on $\Delta(\dv)$ and $20k<n.$
        		To check the second condition of~\thref{thm:bef}, we first note that $\frac{p-k+1}{n-k+1}= 1 -\frac{n-p}{n-k+1} \geq 1 - \frac 1k$ for $n> 2k.$ 
        		Therefore,
                        $$\binom{p-1}{k-1} = \binom{n-1}{k-1}\cdot \frac{(p-1)(p-2) \cdots (p-k+1)}{(n-1)(n-2) \cdots (n-k+1)}
                        \geq \binom{n-1}{k-1} \left(1 -\frac{1}{k} \right)^{k-1} \geq 4\Delta \cdot e^{-1} > \Delta.
                        $$
                        Hence, $\dv$ is $k$-graphical by \thref{thm:bef}.
                        \item  Let $\Delta = \Delta(\dv)$ and $ p = \left \lceil k\Delta^{1/(k-1)}+1 \right \rceil$. Since $(\Delta - 1)p + 1 < 2k \Delta^{1 + 1/(k-1)} \leq dn$ and
        		$$\binom{p-1}{k-1} \geq \left(\frac{k\Delta^{ \frac{1}{k-1}}}{k-1} \right)^{k-1} > \Delta,$$
              \thref{thm:bef}  implies that  $\dv$ is $k$-graphical. \qedhere
                    \end{enumerate}
                \end{proof}
                        
        			For a $k$-graphical sequence $\dv \in \Z^n$ and $\uh \in \binom{[n]}{k}$, let	$P_{\uh}(\dv)$ denote the probability that a random $k$-graph with degree sequence $\dv$ contains the edge $\uh$. We can bound $P_{\uh}(\dv)$  in a simple way using the following switching argument, which is an analogue of the standard graph switching. Denote $\Pid{\uh}=\Pid{\uh}(\dv) = \prod_{v \in \uh} d_v$. Intuitively, the bound given here is reasonable since in the $d$-regular case, the probability of an edge in a given position is   $m/n\choose {q+1}$   by symmetry. Noting that $n=km/d$ gives a formula similar to that below. To match the applications, the following lemma will be stated for $(q+1)$-uniform hypergraphs, instead of $k$-uniform hypergraphs. 
        		\begin{lemma}\thlab{l:simpleSwitching}
        				Let $\dv$ be a  $(q+1)$-graphical  sequence of length~$n$ with $\sum_{v \in V} d_v=(q+1)m$ and $\Delta(\dv) = \Delta$. Assume $\left((q+1)^2 + \Delta^{1/q} \right) \Delta  \leq \frac 18 m $. Then for any $\uh \in \binom{V}{q+1}$ we have
        			$$
         			P_{\uh}(\dv)\le \frac{2 q!\Pid{\uh}(\dv) }{(m(q+1))^q}. 
        			$$
        		\end{lemma}
        		\begin{proof} 
        		Let $\uh_0 = (v_{00}, v_{01}, \dots, v_{0q})$ be an ordering of the elements (vertices) of $K$. We say that a $(q+1)$-tuple is an (ordered) edge in a $(q+1)$-graph $G$ if the corresponding set is an edge in $G$. For each $(q+1)$-graph $G$ with degree sequence $\dv$ containing $\uh_0$, we can perform a switching by removing $\uh_0$ and $q$ other ordered mutually disjoint edges in $G$, $\uh_i = (v_{i0}, v_{i1}, \dots,  v_{iq} )$,  and inserting the ordered edges $J_j = (v_{0j}, \dots, v_{qj} )$ for $0 \leq j \leq q$, as long as no multiple edges are formed.
        		    We will show that the number of suitable choices for the ordered edges $\uh_i$ is at least 
        			\bee \label{eq:simplesw}
        			 \left( (m-(q+1)^2 \Delta)(q+1)! \right)^{q} - (q+1) \left(q!\Delta \right)^{q+1}  \geq  (m(q+1)!)^q \left(1 - \frac{(q+1)^2 \Delta}{m} - \frac{\Delta^{q+1} (q+1)!}{(m(q+1))^q} \right).
        			\ee
        			For, having chosen $\uh_0, \dots, \uh_i$ for some $i < q$, there are at most $(q+1)^2\Delta$ edges with a vertex in $\bigcup_{i' = 0}^{i}\uh_{i'}$, which are unavailable as a choice for $\uh_{i+1}$. Moreover, we need to subtract choices for $J_i$ which are already an edge of $G$. The number of such choices for the $(q+1)$-tuple $J_0 = (v_{00}, v_{10}, \dots v_{q0})$ which forms an edge in $G$ is at most $\Delta q!$. Given these vertices, the number of ways to choose and order the remaining edges $K_i$ containing $v_{i0}$ is at most $ (\Delta q!)^q$. Repeating this for $J_1, \dots, J_q$ gives $(q+1)(\Delta q!)^{q+1}$. By assumption, $\left((q+1)^2 + \Delta^{1/q} \right) \Delta  \leq \frac 18 m $, so the right-hand side of~\eqref{eq:simplesw} is at least		
         $$ \frac 12 (m(q+1)!)^q.$$
         
        			On the other hand, for each $k$-graph $G'$ in which  $\uh$ is {\em not} an edge, the number of ways that $\uh$ is created by performing such a switching in reverse is at most $\prod_{v\in \uh} d_{v} (q!)^{q+1} = \Pid{\uh}(\dv)(q!)^{q+1}$. To see this, note that each $J_j$ has to contain $v_{0j}$ as the first vertex in the ordering, and the remaining vertices can be ordered arbitrarily.
        			
        			Counting the set of all possible switchings over all such $(q+1)$-graphs $G$ and $G'$ in two different ways shows that the ratio of the number of graphs with  $\uh$  to the number without  $\uh$ is
        $$
        \frac{P_{\uh}(\dv)}{1-P_{\uh}(\dv)} \leq \frac{2 q!\Pid{\uh}(\dv) }{(m(q+1))^q }.
        $$
        In particular, $P_{\uh}(\dv)\le \frac{2 q!\Pid{\uh}(\dv) }{(m(q+1))^q} $, as required.
        		\end{proof}
        
        		\subsection{Concentration results}
        		    The following lemma is originally due to McDiarmid~\cite{McD}, and the present formulation can be found in~\cite{lw17}.
                    \begin{lemma}[McDiarmid]
                    \thlab{l:subsetConc}  Let $c>0$ and let $f$ be a function  defined on the set of   subsets of some set $U$ such that $\size{f(S)-f(T)}\le c$ whenever $|S|=|T|=m$ and $|S\cap T|=m-1$. Let $S$ be a  uniformly random $m$-subset of $U$.   Then   for all $\alpha>0$ we have
                    $$
                    \pr\left(\size{f(S)-\ex f(S)} \ge \alpha c\sqrt m  \right) \le 2\exp (- 2\alpha^2).
                    $$
                     \end{lemma}
                      Recall that by $\omega(1)$ we denote a function that tends to $\infty$ arbitrarily slowly with $n$, and that $\Bknm$ is  a sequence of $n$ i.i.d.~random variables each distributed as ${\rm Bin}\left( \binom{n-1}{k-1}, \frac 12 \right)$ conditioned on $\sum d_v=km$. 
  \begin{lemma}
                    \thlab{l:sigmaConc}  Define $\dv=(d_1,\ldots, d_n)$ as either (a) the degree sequence of a random $k$-graph in $\Hknm$, or   (b)  a sequence in $\Bknm$. Let $d=km/n$.  
                    \begin{enumerate}[(i)]
                    \item \label{itm:degconc} For $1\le v\le n$ and  all $\alpha>0$   we have 
                    $$
                    \Pr(|d_v-d|\ge \alpha  )\le 2\exp\bigg(-\frac{\alpha^2}{2(d+\alpha/3)}\bigg) \leq 2 \exp \left(- \min \left \{ \frac{\alpha^2 }{3d}, \frac{\alpha}{3} \right \} \right).
                    $$ 
                    \item \label{itm:sigmaconc} If $\beta k \geq 100 $, then
                     we have 
                          $$     \pr\left( \size{\sigma^2 (\dv) - \Var\, d_1 }  \ge  2d  \left( \frac{\beta k \log n}{\sqrt{n}}  + \sqrt{ \frac{(\beta k \log n)^3}{nd} } \right) \right) =O\left( n^{-\beta/10} \right) .$$
                    Moreover, $ \Var\, d_1 = d\left(1- m\binom{n}{k}^{-1}\right)\cdot \rel{ \frac kn }.$ 
                    \end{enumerate}
                    
                    \end{lemma}
                    \begin{proof} 
          First consider (a). In $\Hknm$, each  vertex degree $d_v$ is distributed hypergeometrically with parameters ${n\choose k},m, {n-1 \choose k-1}$. Thus the expected value of $d_v$ is $d $, and hence  (i) holds by~\cite[Theorem~2.10, equations~(2.5) and~(2.6)]{jlr}.

                                        To show (ii), 
                     define
                    $$
                     f(\dv) = f(\dv(G)) =\sum_{v \in [n]} \min\{(d_v-d)^2,  x\} 
                    $$
                   for a $k$-graph $G$ with degree sequence $\dv = \dv(G)$,   
                    where $x = \beta k d \log n \left(1 + \frac{\beta k \log n}{4d} \right)$. 
                    A $k$-graph $G \in \Hknm$ corresponds to a random $m$-subset of all possible edges. {Moving} one edge in $G$ {to a different position} affects at most $2k$ vertex degrees, and the contribution of each vertex to the change {in}  $f$ is at most $ (\sqrt{ x})^2 - (\sqrt{x}-1)^2 <2\sqrt x$. Thus  Lemma~\ref{l:subsetConc} applies with $c=4k\sqrt x$. Taking $\alpha^2 = \frac{\beta \log n}{16}$ and $m = \frac {dn}{k}$, we get
                    \begin{align*}
                    \pr \left( |f(\dv)-\ex f(\dv)|  \ge  \sqrt{ \beta x d kn \log n}  \right) &\le 2\exp \left(- \frac{\beta \log n}{8} \right).
                    \end{align*}
                    
                    On the other hand, 
                    let    $A$ denote the event $\max_v|d_v-d|\ge \sqrt x$. The parameter $x$ was set so that $\frac{x}{3d}\geq \frac{\beta k \log n}{3}$ and $\frac{\sqrt{x}}{3} \geq \frac{\beta k \log n}{2 \cdot 3}$. Therefore, by (i) and the union bound,
                    $$\pr(A) \leq 2n \exp \left( - \min \left \{ \frac{x}{3d}, \frac{\sqrt{x}}{3} \right \} \right) \leq
        2n^{-\beta k/ 6+1} =O \left( n^{-\beta k /10} \right).$$  
                    Provided that $A$ does not hold,  we have $f(\dv)= \sum_v (d_v-d)^2=n \sigma^2(\dv)$. We thus conclude 
                    $$
                    \pr\left( \size{ n \sigma^2 (\dv) -\ex f(\dv)}  \ge  \sqrt{ \beta x d kn \log n}  \right) \leq 2 n^{-\beta /8 } + \pr(A) = O \left(n^{-\beta/10} \right).
                    $$
                    To pass from $\ex f(\dv)$ to $\ex \sigma^2(\dv) = \Var  \, d_1$, we note that for all $\dv$, $|f(\dv) - \sigma^2(\dv)| \leq 2\sigma^2(\dv) \leq n^{2k-1}$ and thus, applying this crude upper bound whenever $f(\dv)\ne  n \sigma^2(\dv)$,
                    $$
                    \size{\ex \left[ f(\dv) - n \sigma^2(\dv)  \right] } \le 2n^{2k-1}  \pr(f(\dv)\ne  n \sigma^2(\dv) )= O \left( n^{2k-1 - \beta k / 6} \right) = o\left( n^{-1} \right). 
                    $$
                   It follows that 
                    $$
                    \pr\left( \size{\sigma^2(\dv) - \Var\, d_1 }  \ge  \sqrt{ \frac{ \beta x d k \log n}{n} } + \frac 1n \right) =O\left( n^{-\beta/10} \right) .
                    $$
                   It remains to check that the deviation $\sqrt{\beta x d k \log n / n} + n^{-1}$ is smaller than the one claimed in (ii). Indeed,
                   $$ \sqrt{ \frac{ \beta x d k \log n}{n} } + \frac 1n \leq 2 \sqrt{\frac 1n \left( (\beta k d \log n)^2 + d(\beta k \log n)^3 \right) } 
                   \leq \frac{d}{\sqrt n} \left( \beta k \log n + \sqrt{\frac{(\beta k \log n)^3} {d}} \right).
                   $$
                   This concludes the proof of part (ii) for (a). The estimate for $\Var\,  d_1$ follows from the standard formula for variance of this hypergeometric random variable.
                    
                    For the binomial random variable case (b), {a similar  argument applies} 
                     by modelling $ d_1,\ldots, d_n $ each as a sum of $\binom{n-1}{k-1}$ independent indicator variables conditioned on $\sum_{v \in [n]}d_v =  km.$ Conditioning on the sum  being $km$ is equivalent to a uniformly random selection of a $(km)$-subset of the $n \binom {n-1}{k-1}$ indicator variables.  
                     In particular, every $d_v$ is distributed hypergeometrically with parameters $n\binom{n-1}{k-1}, km, \binom{n-1}{k-1}$,  which leads to both (i) and (ii) by an argument virtually identical to that for (a). 
        \end{proof}

     \section{Recursive relations} \label{sec:recursions}    
            We will now derive recursive relations which, when iterated, allow us to estimate the probabilities of adjacent degree sequences with increasing accuracy. Recall that we are considering $k$-uniform hypergraphs   on the vertex set $V$ which has size $n.$  
            For $\dv \in \Z^n$  and $a, b \in V$  with $\Num(\dv - \eb) >0$, we define
                $$R_{ab}(\dv) = \frac{ \pr_{\Dknm}(\dv - \ea) }{ \pr_{\Dknm}(\dv - \eb)} = \frac{\Num(\dv - \ea)}{\Num(\dv - \eb)}.$$
                     For $A \in \binom{V}{k-1}$ and $v \notin A$, we set $v + A = A+v = A \cup \{v \}$. 
                 For an integer $k$, sets $J, L \in \binom{V}{k}$ with $J\neq L$ and $\dv \in \Z^n$, let $\Num_{L}(\dv)$ denote the number of $k$-graphs with degree sequence $\dv$ containing the edge $L$; and let $\Num_{J, L}(\dv)$ be the number of $k$-graphs with degree sequence $\dv$ containing both edges $J$ and $L$. Before stating the recursive relations, we state a simple identity from~\cite{lw17} which will be used several times to relate degree sequences of differing total degree. It was proved in~\cite{lw17} for $k=2$, but the proof readily extends to arbitrary $k$. 
                                  
                \begin{lemma} \thlab{trick17}
                Let $J, L \in \binom{V}{k}$ and let $\dv$ be a sequence of length $n$. 
                Then 
                \begin{align*}
                \Num_{J, L}(\dv) 
                	&= \Num_J (\dv - \ve_{L}) -  \Num_{J,L} (\dv - \ve_{L}) \quad \text{and, similarly,}\\
                \Num_{L}(\dv) 	&= \Num (\dv - \ve_{L}) -  \Num_{L} (\dv - \ve_{L}).
                \end{align*}
                \end{lemma}

             The recursive relations include ordered $(k-1)$-tuples, so let us specify an ordering which breaks ties and is easy to work with. Given $A, B \in \binom{V}{k-1}$, the \emph{A-consistent} ordering of the pair $(A, B)$ is an ordering $(a_1, a_2, \ldots, a_{k-1})$ of $A$ and $(b_1, \ldots, b_{k-1})$ of $B$ 
             such that (1) $a_i < a_{i+1}$ for $i \in [k-2]$, (2) whenever  $b_i \in A$, then $b_i = a_i$, and (3) if  $b_i, b_j \notin A$ and $i<j$, then $b_i < b_j$. For $A, B 
              \in \binom{V}{k-1}$ and their $A$-consistent ordering $(a_j)_{j \in [k-1]}, (b_j)_{j \in [k-1]}$, define
                \begin{equation}
                    \label{eq:ratioprod} R_{\ub, \ua} (\dv)= \prod_{j=1}^{k-1} R_{b_j a_j} (\dv  - \ve_{b_1b_2\cdots b_{j-1}a_{j+1}\cdots a_{k-1}}) ,
                \end{equation}
             assuming that all the sequences $\dv -\ve_{b_1 \cdots b_j a_{j+1} \cdots a_{k-1}}$  are $k$-graphical. 
         
            Let $K \in \binom{V}{k-1}$ and $a, b \in V$ with $a \neq b$. We set $\sp_{a, K, a}(\dv) = 0$ and define $\sp_{a, K, b} (\dv)$ as the probability that a random $k$-graph with degree sequence $\dv$ contains the edges $a+K$ and $K+b$, so we have
                \begin{equation}    \label{eq:psdefinitions}
                    \sp_{a, K, b} = \frac{\Num_{a+K, K+b}(\dv)}{\Num(\dv)} \quad \text{and} \quad P_{a+K}(\dv) = \frac {\Num_{a+K}(\dv) }{\Num(\dv)}. 
                \end{equation}
            
             We now state the identities relating $R_{ab} (\dv)$, $P_{a+ K} (\dv)$ and $\sp_{a, K, b} (\dv)$. They extend the relations from~\cite{lw17} which concern graphs. 
               \begin{lemma}  \thlab{t:recrel}
                Let  $2\le k \le n$ be integers  and $\dv \in \Z^n$. 
                \begin{itemize}
                    \item[(a)]
                        For $a, b \in V$, if $\dv-\ea$ and $\dv - \eb$ are $k$-graphical, then
                    	\begin{align} \label{eq:Rfixed}
                    	 R_{ab} ( \dv) &= \frac{d_{a}}{d_{b}}\cdot \frac{ 1-\Bad(a,b, {\bf d} - \eb)}{ 1-\Bad(b,a, {\bf d} - \ea)}, 
                    	\end{align}
                    	provided that  $\Bad(b,a, {\bf d} - \ea) < 1$,	where 
                    	\begin{align}\label{eq:bad}
                    	\Bad(i,j, \dw) 
                    	& = 
                    	 \frac{1}{d_{i}}\left(\sum_{ \uh \in \binom{V \setminus \{ a, b\}}{k-2} }  P_{i + \uh+ j}(\dw)   +
                    	 \sum_{ \uh \in \binom{V \setminus \{ a, b\}}{k-1} }  \sp_{i, \uh, j}(\dw) \right). 
                    	\end{align}
                    \item[(b)] Let $\ua \in \binom{V}{k-1}$ and $v \in V \setminus \ua$. We have 
        	\begin{equation}  \label{eq:Pfixed} 
        	P_{\ua\unn v}(\dv) = d_{v} \left( \sum_{\ub \in \binom{V \setminus \{ v\}}{k-1}} R_{\ub, \ua} (\dv-\ve_{v})
        	\frac{1-P_{\ub\unn v}(\dv - \ve_{\ub \unn v} )}
        	{1-P_{\ua \unn v}(\dv - \ve_{ \ua  \unn v} )} \right)^{-1},
        	\end{equation}
        	    provided that $P_{\ua+ v }(\dv)>0$ and the sequences $\dv -\ve_{b_1 \cdots b_j a_{j+1} \cdots a_{k-1}}$ are $k$-graphical for $j \in \{0, 1, \dots, k-1 \}$  and any $B \in \binom{V \setminus \{ v\}}{k-1}$. 
                    \item[(c)]   For $\uh \in \binom{V}{k-1}$ and $a,  b \in V \setminus \uh$ with $a \neq b$, if $P_{\uh+ a}(\dv-\ve_{\uh+ a}  )< 1$, then
        		\bee \label{eq:Sfixed}
        			\sp_{a, \uh, b } (\dv) = \frac{P_{\uh+ a }(\dv)}{1-P_{\uh+ a}(\dv-\ve_{\uh+ a})} \left(P_{\uh+ a}(\dv - \ve_{\uh+ a}) - \sp _{a, \uh, b}(\dv - \ve_{\uh+ a}) \right).
        		\ee
                \end{itemize}
            \end{lemma}
            \begin{proof}
                 Let $\cH(\dv)=\cH(\dv;k,n)$ be the set of $n$-vertex $k$-graphs with degree sequence $\dv$. Observe that in case $a=b$, (a) holds since $R_{aa}(\dv) = 1$. To establish (a) for $a \neq b$, we analyse a natural \emph{switching} from $\cH(\dv-\eb)$ to $\cH(\dv -\ea)$. Let $L_{ab}(\dv)$ be the set of tuples $(G,K)$ where $G \in H(\dv - \eb)$ and $K$ is a $(k-1)$-subset of $V(G)\sm\{a,b\}$ such that $K+a$ is an edge of $G$ and $K+b$ is not an edge of $G.$ In this proof it is convenient to identify a hypergraph $G$ with its edge set.
                For a pair $(G,K)$ of $L_{ab}(\dv)$, switching the edge $K+a$ to the edge $K+b$ (that is, replacing $K+a$ by $K+b$) yields a pair $(G \setminus \{ K+a\} \cup \{K +b \}, K)$, which is an element of $L_{ba}(\dv)$. This perturbation, also called degree-switching, clearly yields a bijection between $L_{ab}(\dv)$ and $L_{ba}(\dv)$,
                    which implies $|L_{ab}(\dv)|  = |L_{ba}(\dv)|.$
                We claim that 
                \bee    \label{eq:Lab}
                   |L_{ab}(\dv)| = d_a \Num(\dv- \eb)(1- B(a, b, \dv - \eb)),
                \ee
                where $B(i, j, \dv')$ is defined in ~\eqref{eq:bad}.
                To see this, let $G \in \cH(\dv-\eb)$ be chosen uniformly at random and let $J+a$ be a random edge containing $a$, which we refer to as \emph{distinguished}. The probability measure in this space is denoted by $\pr$. Let $\Ap$ be the event  that $b \in J$ and $\Ae$ the event that $J+b$ is an edge of $G$. Since $\Ae$ and $\Ap$ are disjoint, we have
                $|L_{ab}(\dv)| = d_a \Num(\dv-\eb) (1- \pr(\Ae) - \pr(\Ap))$ and it suffices to show that $\pr(\Ap) + \pr(\Ae) = B(a, b, \dv - \eb)$.
                For any $K \in \binom{V \setminus \{ a, b \}}{k-1}$, we have 
                $$\pr \left( K+b \in G \text{ and } J=K \right) = \frac{\Num_{a+K, K+b}(\dv-\eb)}{\Num(\dv-\eb)d_a} = \frac{{Y_{a,K,b}(\dv-\eb)}}{d_a}.$$
                Summing over the choices of $K$, we get 
                   $$\pr(\Ap) = \frac{1}{d_a} \sum_{K \in \binom{V \setminus \{ a, b \}}{k-1} } Y_{a, k, b}(\dv- \eb).$$
                An analogous argument for the event $\Ae$ implies that 
                $$\pr(\Ae) = \frac{1}{d_a} \sum_{K \in \binom{V \setminus \{ a, b \}}{k-2} } P_{a+K+b}(\dv- \eb),$$
                which completes the proof of~\eqref{eq:Lab}.
               
                Applying~\eqref{eq:Lab} to $|L_{ba}(\dv)|$ and recalling that $|L_{ab}(\dv) |=|L_{ba}(\dv)|$ gives
                $$ R_{ab}(\dv)= \frac{\Num(\dv-\ea)}{\Num( \dv-\eb)} = 
                \frac{d_a}{d_b} \cdot  \frac{ (1- B(a, b, \dv - \eb))}{ (1- B(b, a, \dv - \ea))},$$
               where the denominator is non-zero by the hypothesis of (a).

                For (b), linearity of expectation and simple algebraic manipulation gives
                 $$d_v = \sum_{\ub \in \binom{V \setminus \{ v\}}{k-1} } P_{\ub + v}(\dv) 
                =  P_{\ua + v}(\dv)\sum_{\ub \in \binom{V \setminus \{ v\}}{k-1} } \frac{ \Num_{\ub + v}(\dv) }{ \Num_{\ua + v}(\dv) }.$$
                Using Lemma~\ref{trick17}, it follows that
               	\begin{equation}  \label{eq:prob_derivation}
                	    d_{v}=  P_{\ua\unn v}(\dv)  \sum_{\ub \in \binom{V \setminus \{ v\}}{k-1}} \frac{\Num(\dv-\ve_{B+ v})}{\Num(\dv-\ve_{A+ v})} \cdot
                	\frac{1-P_{\ub\unn v}(\dv - \ve_{\ub \unn v} )}
                	{1-P_{\ua \unn v}(\dv - \ve_{ \ua  \unn v} )} .
                	\end{equation}
                	 For $B\in (V\setminus \{v \} )^{k-1}$, let $(a_1, \dots, a_{k-1}), (b_1, \dots, b_{k-1})$ be the $A$-consistent ordering of $(A, B)$. Since the sequences $\dv - \ve_{\ua +v}$ and $\dv - \ve_{\ub +v}$ differ in at most $k-1$ entries, the fraction ${\Num(\dv-\ve_{B+ v})}/{\Num(\dv-\ve_{A+ v})}$ can be written as a \emph{telescopic} product of the ratios of form $R_{b_j a_j}(\dv')$,
                $$ \frac{\Num(\dv - \ve_{\ub+ v})}{\Num (\dv - \ve_{\ua+ v})} = \prod_{j=1}^{k-1}  \frac{\Num (\dv  - \ve_{b_1b_2\cdots b_{j-1} b_j a_{j+1}\cdots a_{k-1} v }) }{ \Num (\dv  - \ve_{b_1b_2\cdots b_{j-1} a_j a_{j+1}\cdots a_{k-1}v}) } = R_{\ub, \ua} (\dv-\ve_{v}),
                $$
                recalling that all the relevant degree sequences are $k$-graphical by assumption. Substituting this identity back into~\eqref{eq:prob_derivation}, we get the relation (b).
                
                For (c),
                applying \thref{trick17} and recalling equation~\eqref{eq:psdefinitions} defining $P$ and $\sp$ gives		
                \begin{align*} \non
        			\sp_{a, \uh, b } (\dv) {\Num(\dv)} &= \Num_{\uh+b}(\dv-\ve_{\uh, a}) - \Num_{\uh+a, \uh+b}(\dv - \ve_{\uh, a})
        			\\
        		&= \left(P_{\uh, b}(\dv - \ve_{\uh, a}) - \sp_{a, \uh, b}(\dv - \ve_{\uh, a}) \right)	{\Num (\dv - \ve_{\uh, a})} .
        		\end{align*}
        		   Lemma~\ref{trick17} also implies that $\frac{ \Num(\dv - \ve_{\uh, a}) } { \Num(\dv) } = \frac{ P_{\uh, a}(\dv) }{1 - P_{\uh+a}(\dv - \ve_{\uh, a}) } $  whenever $P_{\uh+ a}(\dv - \ve_{\uh, a})  <1$. Substituting this into the previous identity, we get
        		\bee 
        			\sp_{a, \uh, b } (\dv) = \frac{P_{\uh, a }(\dv)}{1-P_{\uh, a}(\dv-\ve_{\uh, a})} \left(P_{\uh, b}(\dv - \ve_{\uh, a}) - \sp _{a, \uh, b}(\dv - \ve_{\uh, a}) \right),
        		\ee
        		as required.
            \end{proof}

        \section{Sparse hypergraphs} \label{sec:sparsecase}

                We now present the proof of~\thref{t:sparseCase}. Under the assumptions of the Theorem, we only need two iterations of the recursive relations for a sufficiently accurate estimate for  $R_{ab}(\dv)$, the ratio between probabilities of two adjacent sequences. 
                Given a sequence $\dv$ we write  $M_1(\dv)=\sum_v d_v$; recall that $\Delta(\dv)=\max_v d_v$.
                \begin{proof}[Proof of \thref{t:sparseCase}]
                  Let
                $$\Delta_1= 2 \Delta_*+ \log^3 n $$
                and define  $\D^+$ 
                to be the set of all sequences $\dv\in \Z^n$ 
                with $\Delta(\dv)\le    \Delta_1$ and $M_1(\dv)=km$.
                 Recall that the $\Lh$-distance was defined in Section~\ref{sec:preliminaries} and that $\Lh(\ve_{\uh}) =1$ for $|K|=k$.  For an integer $ r\ge 0$, denote by $Q^0_{r}$ (or $Q^1_{r}$) the set of sequences $\dv$ in $\Z^n$ that  have $\Lh$-distance at most $r$ from  some sequence in $\D^+$ and satisfy $M_1(\dv) \equiv 0 \pmod{k}$ (or $M_1(\dv) \equiv 1 \pmod{k}$ respectively). Let $\nu = \frac{2\Delta_1}{m} =\frac{2\Delta_1 k}{dn} =  O\left(\frac{\pseta}{\log^2 n} \right)$, recalling the definition of $\pseta$ and the assumption $\pseta < 1$ from the theorem statement. 
                
                 We first  estimate the ratios $R_{ab}(\dv)$ for $\dv \in Q_1^1$ by iterating the recursive relations from \thref{t:recrel}. The following property of the formulae (\ref{eq:Rfixed}\,--\,\ref{eq:Sfixed}) will be important: if the argument is $\dv$, then the sequences referred to in the formula are at $\Lh$-distance at most 1 from $\dv$.

                \begin{claim}\thlab{RSparse}
                Uniformly for  all $\dv \in  Q^{1}_{1}$ and for all $a,b\in [n]$,
                $$R_{ab} (\dv) = \frac{d_a}{d_b}\left(1 + (d_a - d_b) \cdot \frac{k-1}{M_1(\dv)} \right) \rel{(k-1) \nu^2 + \Delta_1 \nu^{k-1}}.$$
                \end{claim}

                 \begin{proof}
                   Since our statement is asymptotic, we assume that $n$ is sufficiently large in our estimates. The hypothesis $\pseta <1$ implies that $m> \log^9 n$.  Throughout the proof, we use the fact that for all $\dv \in Q^0_{ 5} \cup Q^{1}_5$, 
                        \bee    
                            \Delta (\dv) \leq \Delta_1 + 5 
                            \quad \text{and} \quad |M_1(\dv) - km| \leq 5k.
                        \ee 
                        In particular, 
                        \bee  \label{eq:nu}
                        	 \frac{k \Delta(\dv)}{M_1(\dv)} \leq \frac{2 \Delta_1}{m}= \nu.
                        \ee
                        
                The following notation and simple inequalities will be useful.	Recall that $\Pid{\uh} (\dv)= \prod_{v \in \uh} d_v$ for $K \subset V$ and let $\tilde{M}_{j} (\dv) = \sum_{\uh \in \binom{V}{j}} \Pid{\uh} (\dv)$. Specifically, $M_1(\dv) = \tilde{M}_1(\dv) = \sum_{v \in V} d_v$. When the argument~$\dv$ is clear from the context, we suppress it from the notation.
                    We have
                    \bee
                        \lab{eq:moments}
                        \tilde{M}_j \leq \frac 1j \tilde{M}_{j-1} \tilde{M}_1    \text{\quad and \quad} \tilde{M}_j = \frac 1j \tilde{M}_{j-1}\tilde{M}_1 \relf{j \Delta}{\tilde{M}_1}.
                    \ee
                    To see this, simply expand the right-hand side and absorb the monomials which are quadratic in $d_v$ for some $v$ into the error term to obtain
                    \begin{align*}
                        \tilde{M}_1 \tilde{M}_{j-1} & = \left( \sum_{v \in V} d_v \right) \sum_{\uh' \in \binom{V}{j-1}} \Pid{\uh'} (\dv)
                                     =j \sum_{\uh \in \binom{V}{j}}\Pid{\uh}(\dv) + O(j \Delta)\sum_{\uh' \in \binom{V}{j-1}} \Pid{\uh'}(\dv).
                    \end{align*}
                    Since two vertices, say $a$ and $b$, sometimes need to be omitted from the summation, we introduce
                    $$\tilde{M}_j(\dv; a, b) = \sum_{\uh \in \binom{V \setminus 
                \{a, b \}}{j}} \Pid{\uh}(\dv) \quad \text {and} \quad \tilde{M}_j(\dv; a) = \tilde{M}_j(\dv; a, a).$$
                Note that
                \bel{eq:mab}
        	 \tilde{M}_j(\dv; a, b) = \tilde{M}_j(\dv) \relf{\Delta \tilde{M}_{j-1}}{\tilde{M}_j} 
	 =\tilde{M}_j(\dv)\rel{\nu}.
        	\ee

                We start by showing that any sequence $\dv' \in Q^0_{5}$ is $k$-graphical. This follows from Proposition~\ref{prop:kgraphical} by noting that $M_1(\dv') \geq  km-5k$  and
                $$\frac{k \Delta (\dv')^{1+ 1/(k-1)}}{M_1(\dv')} = O\left(\frac{ \Delta_1^2 }{m} \right) = O(\pseta) .$$
                 Therefore, the $k$-graphicality hypotheses of~\thref{t:recrel} are satisfied whenever it is applied below.

                 For brevity, we make the substitution $q=k+1$. Let $\dv' \in Q^0_{ 5}$, so that $|M_1(\dv') - (q+1)m| \leq 5(q+1)$. {For $K \in \binom{V}{q+1}$ and large $n$, we have
        			\begin{equation} \label{eq:pswitch}
        			P_\uh(\dv') \leq 2 \frac{q!\Pid{ \uh }(\dv') }{M_1(\dv')^	q} 
        			\leq 2\Delta(\dv') \left( \frac{q\Delta (\dv')}{M_1(\dv')} \right)^q
        			\leq 4\Delta_1 \nu^q
        			\leq 4\Delta_1 \nu^2
        		    = \frac{16\Delta_1^3}{m^2} 
        			\leq \pseta,
        			\end{equation}
        			by  Lemma~\ref{l:simpleSwitching}, definition of $D_K$,~\eqref{eq:nu}, monotonicity in $q$, definitions of $\Delta_1$, $\nu$ and $\pseta$, respectively.}
                
        			Now consider $\dv \in Q_{ 4}^0$. We derive an upper bound on $B(a, b, \dv)$, which was defined in~\eqref{eq:bad}.
        			For $K \in \binom{V}{q}$ and any distinct $a, b \notin \uh$, $P_{K+a}(\dv-\ve_{K+a})< \frac{1}{2}$ by~\eqref{eq:pswitch}, so
        			$$\sp_{a, \uh, b} (\dv) 
        			\leq 2 P_{K+a} (\dv) P_{K+b}(\dv - \ve_{K+a}) 
        			\leq  \frac{8 d_a d_b \Pid{\uh}^2 (q!)^2}{M_1^{2q}}, $$
        			by Lemma~\ref{l:simpleSwitching}.
        			 From~\eqref{eq:moments} with the sequence $(d_v^2)_{v \in V}$, we obtain the inequality
        		$$  q! \sum_{\uh \in \binom{V}{q}} \Pid{\uh}^2 \leq \left(\sum_{v \in V} d_v^2 \right)^q \leq (\Delta (\dv) M_1 (\dv) )^q .
        $$      
       Applying this to the above bound on $\sp_{a, \uh, b}(\dv)$ gives
        			$$ \frac{1}{d_a} \sum_{\uh \in \binom{V \setminus \{a, b\}}{q} } \sp_{a, \uh, b} (\dv) 
        			\leq  \frac{8 d_b q!}{M_1^{2q}} \cdot \left(\Delta M_1 \right)^{q} 
        			\leq 8 \Delta  \cdot \frac{\Delta^{q}q!}{M_1^q} 
        			=  O \left( \Delta_1 \nu^q \right).$$
        	    We use~Lemma~\ref{l:simpleSwitching} and~\eqref{eq:moments} once again to obtain 
        	\bee    \label{eq:spbound}
        			\frac{1}{d_a} \sum_{\uh \in \binom{V \setminus \{a,b \}}{q-1}} P_{a+ \uh +b}(\dv) \leq \frac{2 q!}{d_a M_1^q} \sum_{\uh \in \binom{V}{q-1}} d_a d_b \Pid{\uh} (\dv)\leq \frac{2 \Delta q}{M_1} = O(\nu).
        	\ee
        			 From~\eqref{eq:bad} and the previous two estimates, it follows that $\Bad(a, b, \dv) = O(\Delta_1 \nu^q + \nu) $. Thus, $\Bad(a, b, \dv) =O(\nu)$ since the assumption $\pseta <1$  implies  $\Delta_1^2 <m$.    Using identity~\eqref{eq:Rfixed}, 
        			we get
        			$$R_{ab} (\dv) 
        			= \frac{d_a}{d_b}\rel{\nu}$$
                for $\dv \in Q_{3}^1$.  
        
        			Now we apply~\eqref{eq:Pfixed} to obtain a more precise estimate for the probability $P_{\ua+ v}$. Let $\dv \in Q_{2}^0$, $\ua, \ub \in \binom{V}{q}$ and $v \in V\setminus \ua$, and let $(a_j)_{j \in [k-1]}$, $(b_j)_{j \in [k-1]}$ be the $A$-consistent ordering of $\ua$ and $\ub$. The above estimate for the ratio $R_{ab}$ only depends on $d_a$ and $d_b$. Moreover, in the expression for $R_{\ub, \ua}(\dv - \ev)$ in~\eqref{eq:ratioprod}, any of the sequences $\dv - \ev -\ve_{b_1\cdots b_{j-1}a_{j+1}\cdots a_{q}}$ is at $\Lh$-distance at most 1 from $\dv$, i.e.~in $ Q_3^1$. Hence  $R_{\ub, \ua} (\dv - \ve_v)= \frac{d_{b_1}d_{b_2} \cdots\, d_{b_q}} { d_{a_1} d_{a_2} \cdots\, d_{a_q} }\rel{ q\nu }$. Thus, by~\eqref{eq:Pfixed}, the probability satisfies
        		\bee \label{eq:Psparse}
        			P_{\ua \unn v} (\dv) = d_v \left(\sum_{\ub \in \binom{V -v}{q}} \frac{\Pid{\ub}(\dv) }{ \Pid{\ua}(\dv) } \rel{q \nu } \right)^{-1}
        				= \frac{\Pid{\ua + v} (\dv)}{ \tilde{M}_q (\dv; v)} \rel{q \nu}.
        		\ee
        			We next return to $\Bad(a, b, \dv)$ introduced in~\eqref{eq:bad}. For the path probabilities, the bound  established in~\eqref{eq:spbound} is sufficient. Let $S_2 = V \setminus \{a, b\}$.  Using~\eqref{eq:Psparse}, we get that for $\dv \in Q_{1}^1$,
        			
        			\begin{align*}
        			    \frac{1}{d_a} \sum_{\uh \in  \binom{S_2}{q-1}} P_{a+ \uh + b} (\dv - \ve_b) &= \frac{(d_b-1) }{\tilde{M}_q(\dv - \ve_b; b)} \sum_{\uh \in \binom{S_2}{q-1}} \Pid{\uh} (\dv - \eb )\rel{q \nu}.  
        			\end{align*}
        			Now
                \bean
                \sum_{\uh \in {\binom{S_2}{q-1}}} \Pid{\uh}(\dv - \eb)  = \tilde{M}_{q-1}(\dv - \ve_b; a, b) 
                 =\tilde{M}_{q-1}(\dv - \ve_b)\rel{\nu} 
               = \frac{ q\tilde{M}_{q}(\dv - \ve_b)}{ \tilde{M}_{1}(\dv )} \rel{\nu},
                \eean
                by definition of $\tilde{M}_{q-1}$,~\eqref{eq:mab} and~\eqref{eq:moments}.  Recalling that $\tilde{M}_1(\dv) = M_1(\dv)$, we have 
            $$ \frac{1}{d_a} \sum_{\uh \in {\binom{S_2}{q-1}}} P_{a+ \uh, b} (\dv - \ve_b)  = \frac{(d_b-1)q}{{\tilde{M}_1}(\dv)}\rel{q \nu} = \frac{(d_b-1)q}{M_1(\dv)} + O(q \nu^2) . $$
        Thus, $B(a, b, \dv - \eb) = \frac{(d_b-1)q}{M_1(\dv)} + O(q \nu^2+\Delta_1 \nu^q)  $ and the analogous estimate holds for $B(b,a, \dv - \ea),$ Finally, by~\eqref{eq:Rfixed}, for $\dv \in Q^1_{1}$,
                \begin{align*}
                    R_{ab} (\dv) &= \frac{d_a}{d_b} \cdot \frac{1- (d_b-1)q/M_1( {\dv-\eb}) }{1 - (d_a-1)q/M_1 ({\dv-\ea})}\rel{q \nu^2 + \Delta_1 \nu^q}
                    \\
                    &= \frac{d_a}{d_b}\left(1 + (d_a - d_b) \cdot \frac{q}{M_1(\dv)} \right) \rel{q \nu^2 + \Delta_1 \nu^q} .  \quad \qedhere
                \end{align*}
                \end{proof}
               Having estimated the ratios between probabilities of adjacent degree sequences, we return to the proof of~\thref{t:sparseCase}. Let $d = km/n$. We aim to show that $\pr_{\Dknm}(\dv)\sim H(\dv)$ for $\dv \in \D^+$, where 
                $$H(\dv)=\pr_{\Bknm}(\dv) \widetilde H(\dv) \text{\quad and  \quad}\widetilde H(\dv) = \exp\left( \frac{k-1}{2}\left(1 - \frac{ \sigma^2(\dv)}{d} \right) \right).$$
               
               To this end, the estimate for $R_{ab}(\dv)$ is used in an argument rather similar to that of Lemma 2.1 in~\cite{lw17}. Unfortunately, a technicality specific to $k$-graphs (more precisely, the fact that the correction term $\tilde{H}(\dv)$ is unbounded in $\D^+$) does not allow a   black-box  application of the existing lemma. For the remainder of the proof, we only consider degree sequences in $\D^+$. The first step is to establish that whenever $\dv-\ea$ and $\dv-\eb$ are elements of $\D^+$, then
                \bel{eq:target1}
                \frac{\Pr_{\Dknm}(\dv-\ea)}{\Pr_{\Dknm}(\dv-\eb)} =e^{O(\delta)}\frac{H(\dv-\ea)}{H(\dv-\eb)}
                \ee
                with
                $$\delta = 
                \begin{cases}
                    \frac{\Delta_1^3}{m^2}, & k = 3, \\
                    \frac{k \Delta_1 ^ 2}{m^2}, & k \geq 4,
                \end{cases} $$
                where  the constant implicit in $O()$ is independent of $\dv$.  Observe that $\pseta > km\delta$, so $\delta=o(1)$.  To simplify the right hand side of~\eqref{eq:target1}, we use the definition $\sigma^2(\dv) = \frac 1n \sum_{v \in V}(d_v-d)^2$ with $d = km/n$  and observe that         
                \begin{align} \non
                 \frac{H(\dv-\ea)}{H(\dv-\eb)}&=   \frac{d_a \left(\binom{n-1}{k-1}- d_b+1 \right)}{d_b \left( \binom{n-1}{k-1} -d_a + 1 \right)}\exp\left( \frac{k-1}{2dn}(2d_a - 2d_b) \right) \\
                \non
                &= \frac{d_a }{d_b } \exp\left( \frac{(k-1)(d_a - d_b)}{dn} \right) \relf{\Delta_1}{ \binom{n-1}{k-1} } \\
               & = \frac{d_a}{d_b} \left( 1+ \frac{(k-1)(d_a-d_b)}{dn} + O\left( \delta \right) \right)     \label{eq:ratioSparse}
                \end{align}
                 for all  $\dv\in Q_1^1$ 
                by~\eqref{eq:condbinom}
                , where we used
                $\exp \left(\frac{k-1}{dn}(d_a - d_b) \right) = 1+ \frac{(k-1)(d_a-d_b)}{dn} + O\left( \frac{\Delta_1^2}{m^2}\right)$ {(recalling that $\Delta_1 = o(m)$)} and $\frac{\Delta_1}{\binom{n-1}{k-1}} = O( \delta)$;  {the latter is   immediate for $k=3$, whilst if $k>3$ then $\Delta_1 / n^3= \Delta_1 d^3n^{-3}k^{-3} = O((\Delta_1^2/m^2)\cdot(\Delta_1^2/m))=O(\delta)o(1)$ considering that $\psi_k=o(1)$ (see~\eqn{eq:sparse4}).}   
                 Combining \eqn{eq:ratioSparse} with \thref{RSparse} and noting that the error from Claim~\ref{RSparse} satisfies $q\nu^2 + \Delta_1 \nu^q = O(\delta)$, we conclude that whenever $\dv-\ea$ and $\dv - \eb$ lie in $\D^+$,
                 \bee \non
                 \frac{  \pr_{\Dknm}(\dv-\ea)}{ \pr_{\Dknm}(\dv-\eb) } = R_{ab}(\dv) = e^{O(\delta)}\frac{  H(\dv-\ea)}{ H(\dv-\eb)}.
                \ee
              Thus we have~\eqref{eq:target1}.

               If $d$ is an integer, let $\dv_{\reg}  = (d, d, \dots, d)$. Note that, in this case,  $ \dv_{\reg}\in \D^+$ and, crucially, 
    for any sequence $\dv \in \D^+$, there is a sequence of sequences $\dv=\dv_1,\dv_2,\ldots, \dv_r=\dv_{\reg} \in \D^+$ with $r \leq km$ in which consecutive sequences are adjacent.  
                By `telescoping'~\eqref{eq:target1} {and using $km\delta < \pseta < 1$}, we deduce that for $\dv \in \D^+$,
              \bel{eq:telescope}
                \frac{\pr_{\Dknm}(\dv)}{H(\dv)} = c_1 \rel{km \delta},
              \ee
        where $c_1 =  \pr_{\Dknm}(\dv_{\reg})/H(\dv_{\reg})$. 
         On the other hand, if $d$ is not an integer, then  we reach a similar conclusion by {defining  $\dv_{\reg}$ to be an appropriate sequence in} $\{ \lfloor d \rfloor, \lfloor d \rfloor +1\}^n$.
             
             In determining $c_1$ asymptotically, the key ingredient is that the correction term $\tilde{H}(\dv)$ is concentrated around 1 in both $\Bknm$ and $\Dknm$. 
       For later use, we isolate the following claim, which {essentially} implies Theorem~\ref{t:sparseCase}.      
       Recall that $d=km/n$ and that $\Delta_*$ is an upper bound on $\Delta(\dv)$ for a sequence $\dv$ with average $d$, by assumption of \thref{t:sparseCase}. It follows that $d\le \Delta_*.$
       \begin{claim}\thlabel{cor:binapp_sparse}
       Let 
       $$ \W = \left \{ \dv \in \D^+:  \left| \sigma^2(\dv)-d\right| \leq  d {\xi} \right \},$$  
       where   
       $$
       \xi = \frac{k  \log^2 n}{\sqrt{n}} + \sqrt{\frac{k^3 \log^5 n}{nd}} .
       $$
        Then $\pr_{\Dknm}(\W)= 1- O \left(n^{-\omega(1)}\right)$ {and}  $\pr_{\Bknm}(\W) = 1 - O \left(n^{-\omega(1)}\right)$.
        Moreover,  uniformly for $\dv  \in \W$,  
        \bel{new}
        \pr_{\Dknm}(\dv) = \pr_{\Bknm}(\dv) \rel{\sqrt{\pseta}}  .
        \ee
                \end{claim}
            \begin{proof}[Proof of Claim~\ref{cor:binapp_sparse}]
                    Recall that for $\dv \in  \D^+$ we have $M_1(\dv)=km$ and $\Delta(\dv)\leq    \Delta_1=    2 \Delta_*+ \log^3 n$.
               To prove the claim, we first estimate the probability of $\D^+$ in $\Dknm$ by applying~\thref{l:sigmaConc}(i) with $\alpha =  \Delta_* \log^3 n >d \log^3 n$.  By definition of $\Delta_1$, we have that $ \Delta_1 - \Delta_* = \Delta_* + \log^3 n > \sqrt{ \Delta_* \log^3 n}$, so $\Delta_1 \ge \Delta_* + \sqrt {\Delta_*\log^3 n} \geq  d+ \sqrt {\Delta_*\log^3 n}$. 
                Thus, for   $\dv\in  \Omega$ and $v \in V$, 
                $$
                 \pr_{\Dknm}(d_v > \Delta_1)\le  \pr_{\Dknm}\big(d_v>d+ \sqrt {\Delta_* \log^3 n}\big)    = O(n^{- \omega(1)} )
                $$ 
                by~\thref{l:sigmaConc}(i) (and noting $\Delta_*\to\infty$).
                 The union bound, applied to each $v \in V$, now gives $\pr_{\Dknm}(\D^+) =  1- O \left(n^{-\omega(1)}\right)$. 
                        
                Next, if $\dv$ is chosen according to $\Dknm$  then         $\sigma^2 (\dv) = d(1 \pm \xi)$ with probability $1-O \left(n^{-\omega(1)}\right)$ by~\thref{l:sigmaConc}(ii) with  {$\beta = \sqrt{\log n }$}.                 %
            By definition of $\W$ and the union bound, we have
                 $$\pr_{\Dknm}(\W)=  1- O \left(n^{-\omega(1)}\right) .$$
       The same argument shows that $\pr_{\Bknm}(\W) = 1 -O \left(n^{-\omega(1)}\right)$.
           
        It is immediate that for $\dv \in \W$, $\tilde H (\dv) = e^{O(k  \xi)} $. Since $k \xi < \sqrt{\pseta}<1$ for large $n$,
                 \bee
                    \sum_{\dv \in \W} H(\dv) = \rel{k \xi}\sum_{\dv \in \W} \pr_{\Bknm}(\dv) 
                    =  1 - O(n^{-\omega(1) }+ k\xi).
                 \ee
                
               Summing both sides of~\eqref{eq:telescope} over $\W$ and using the above estimates, we get  
        $$
       \pr_{\Dknm}(\W) =\sum_{\dv \in \W} \pr_{\Dknm}(\dv) =  c_1 \rel{km \delta}\sum_{\dv \in \W} H(\dv)  = c_1 \rel{km \delta + k \xi}.
        $$
               From this estimate and $\pr_{\Dknm}(\W) =  1 - O \left(n^{-\omega(1)}\right) $, we deduce that $c_1  = 1 + O \left( km \delta + k \xi + n^{-\omega(1)} \right) = 1 + O \left(\sqrt{\pseta} \right)$, where the last inequality holds by definition of $\pseta$. 
       Hence,~\eqn{eq:telescope} implies that for $\dv \in \D^+$,  
        \bee
            \label{eq:telescope1}
                \pr_{\Dknm}(\dv) = \pr_{ \Bknm } (\dv) \tilde{H}(\dv) \rel{{\sqrt{\pseta} } }.
        \ee
        In particular, for $\dv \in \W \subset \D^+$, $ \tilde{H}(\dv) =  e^{o(\sqrt{\pseta})}$, which implies~\eqn{new} and completes the proof of the claim. 
        \end{proof} %
        The theorem follows from~\eqn{eq:telescope1} considering the definition of $\tilde{H}(\dv)$.
        \end{proof}

    \section{Operators}
                \label{sec:operators}
                Motivated by the recursive relations from Section~\ref{sec:recursions}, we define operators $\Pc, \Rc$ and $\Sc$. The functions $P, R$ and $Y$ are invariant under $\Pc, \Rc$, $\Sc$ in a sense which is formalised later.

                Recall that $V = [n]$, and define $\Vp = \{(\uh, a): \uh \in \binom{V}{k-1}, a \in V \setminus \uh \}$, as well as $\Vs = \{(a, \uh, b): \uh \in \binom{V}{k-1}, a, b \in V \setminus \uh, a \neq b \}$. Let $\pv: \Vp \times \Z^n \to \R$, $\rv: V^2 \times \Z^n \to \R$ and $\sv: \Vs  \times \Z^n \to \R$. We write $\pv_{A,v}(\dv)$ for $\pv(A, v, \dv)$, $\rv_{ab}(\dv)$ for $\rv(a, b, \dv)$ and $\sv_{a,K,b}(\dv)$ for $\sv(a, K, b, \dv)$.  
                
            For $A, B \in \binom{V}{k-1}$, define $$\rv_{\ub, \ua} (\dv)= \prod_{j=1}^{k-1} \rv_{b_j a_j} (\dv  - \ve_{b_1b_2\cdots b_{j-1}a_{j+1}\cdots a_{k-1}}),$$
            with respect to the $A$-consistent ordering of $\ua$ and $\ub$.

              The operators $\Pc$, $\Rc$ and $\Sc$ are defined as follows.  The arguments of the operator $\Pc$ are functions $\pv$ and $\rv$ as above, and its output $\Pc(\pv, \rv)$ is a function $\Vp \times \Z^n  \to  \R$. Let $\dv \in \Z^n$. For $(\ua, v) \in \Vp$, let
        	\begin{equation} \label{eq:probop}
        	\Pc(\pv, \rv)(A, v, \dv) =d_{v} \Bigg(\sum_{\ub \in \binom{V \setminus \{v \}}{{k-1}}} \rv_{\ub, \ua} (\dv-\ve_{v})
        	\frac{1-\pv_{\ub, v}(\dv - \ve_{\ub ,  v} )}
        	{1-\pv_{\ua ,  v}(\dv - \ve_{ \ua  ,  v} )}\Bigg)^{-1}.
        	\end{equation} 
            Similarly, $\Sc(\pv, \sv)$ is a function $\Vs  \times \Z^n \to \R$. 
             For $(a, \uh, b) \in \Vs$, we define 
                            \bee    \label{eq:pathop}
                                \Sc(\pv, \sv)(a, \uh, b  , \dv)= 
                                \frac{\pv_{a ,  \uh}(\dv)}{ 1- \pv_{a ,  \uh}(\dv - \ve_{a ,  \uh})} \left(\pv_{b ,  \uh}(\dv - \ve_{a ,  \uh}) - \sv_{a, \uh, b}(\dv - \ve_{a ,  \uh}) \right).
                            \ee
            Let $\dw \in \Z^n$ and $i \neq j \in V$. We set $\bade(\pv)( i, i, \dw) = \badp(\sv)( i, i, \dw)=0$,
            \begin{align} \label{eq:badop}
            	\bade(\pv)( i, j, \dw)  = 
            	 \frac{1}{d} 
            	 \sum_{ \uh \in \binom{V \setminus \{i, j \} }{k-2}}  \pv_{i + \uh ,  j}(\dw)  
            	 \text{ \quad and \quad } 
            	 \badp(\sv)( i, j, \dw) =
            	\frac{1}{d}\sum_{ \uh \in \binom{V \setminus \{i, j \} }{{k-1}} }  \sv_{i ,  \uh ,  j}(\dw),
        	\end{align}
             Let $d = \frac 1n \sum_{v \in V}d_v$ and $a, b \in V$. The ratio operator is then defined as 
        	\begin{align} \lab{eq:ratop}
        	 \Rc(\pv, \sv)(a,b, \dv) &=  \frac{ d_a - d \left( \bade(\pv)( a, b, \dv - \ve_b)+ \badp(\sv)( a, b, \dv - \ve_b)  \right)}
        	 { d_b -d \left( \bade(\pv)( b, a, \dv-\ve_a)+ \badp(\sv)( b, a, \dv-\ve_a)  \right)}, 
        	\end{align}
        	provided that the denominator in the above formula is non-zero. 
        	Finally, we introduce the compositional operator $$\Cc(\pv, \sv)= \left(\Pc(\pv,\Rc(\pv, \sv)), \Sc(\Pc(\pv,\Rc(\pv, \sv)), \sv ) \right).$$ 
        	The crucial properties of $\Cc$ are that $(P,Y)$ is a fixed point of $\Cc$ whenever the hypotheses of \thref{t:recrel} are satisfied (we prove a precise statement in the next section); and that $\Cc$  is \emph{contractive}, in a sense we are about to formalise.
        	
                For an integer $r$ denote by $Q^0_{r}(\dv)$ (or $Q^1_{r}(\dv)$) the set of sequences $\dv'$ in $\Z^n$ that  have $\Lh$-
                distance at most $r$  from  $\dv $ and satisfy $M_1(\dv) \equiv 0 \pmod{k}$ (or $M_1(\dv) \equiv 1 \pmod{k}$ respectively).
                We use $1\pm \xi$ to denote a quantity between $1-\xi$ and $1+\xi$ inclusively.

                We introduce more notation to avoid some quantifiers. Let $\mathbf{f} : W \times \Z^n \to \R $ be a function, where $W$ is some set. As above, we abbreviate $f(x, \dv) = f_x(\dv)$ for $x \in W$. For any $\dv$ in $\Z^n$, we introduce the `projection'
                    \begin{align}
                        {\bf f }\proj (\dv) :  W & \to  \R \\
                                        x & \mapsto {\bf f}_x(\dv). 
                    \end{align}
                   
                This specifies for instance $\pv \proj (\dv)$ and $\Rc  (\pv, \sv)\proj (\dv)$.
             An (in)equality involving ${\bf f}\proj (\dv)$ means that the (in)equality holds for all $x \in W$. For instance, ${\bf f}\proj (\dv) = (1 \pm \xi) {\bf f}' \proj (\dv)$ if ${\bf f}_x (\dv) = (1 \pm \xi) {\bf f}'_x (\dv)$ for all $x \in W$. The following lemma is an intermediate step in showing the contractiveness of $\Cc$.
             %
        \begin{lemma}\thlab{l:errorImplication} 
        Let $n$ and $k \geq 3$ be integers, $\alpha = \frac{k-1}{n-1} = o(1)$ and $\dv=\dv(n)\in\Z^n$ with $d = \frac{1}{n} \sum_{v \in V}d_v$ and $d_v = d\left(1 \pm \frac 12 \right)$ for all $v \in V$.  
        Let $0<\xi_1, \xi \leq 10$, and write $\mu = d \binom{n-1}{k-1}^{-1}$.
                 Let $\pv,\pv' : \Vp \times \Z^n \to \reals$, $\sv, \sv':
                  \Vs \times \Z^n \to \reals$ and  $\rv, \rv': V^2\times \Z^n \to \reals$
                be functions such that for all $\dv'\in Q_1^0(\dv)$,
                 $\pv \proj (\dv') ,\pv' \proj (\dv') \leq 10 \mu < \frac 14$,
                 $\sv \proj (\dv') ,\sv' (\dv) \proj  \leq 10\mu^2$ and
                 $$ \pv \proj (\dv')=\pv' \proj (\dv')(1 \pm \xi_1 ) \quad \text{and} \quad \sv \proj (\dv')=\sv' \proj (\dv')(1 \pm \xi ).$$
             The following estimates hold {uniformly}.
                  \begin{itemize}
                \item[$(a)$]  
                If $dn \equiv 1 \pmod{k}$,  then $ \displaystyle
                \Rc  (\pv, \sv) \proj (\dv ) = \Rc(\pv', \sv') \proj (\dv)  \rel{\mu \xi +\alpha \xi_1 }
                $.
                \item[$(b)$]  
               If $dn \equiv 0 \pmod{k}$ and $\rv \proj (\dv')=\rv' \proj (\dv')(1 \pm (\mu +\alpha)\xi )$ for all $\dv'\in  Q^1_1 (\dv)$ with $(\mu + \alpha)k \xi < 1$,  then
                \begin{align*}
                \Pc(\pv,\rv) \proj (\dv) =\Pc(\pv',\rv') \proj  (\dv) \rel{ (\mu+\alpha)k\xi  + \mu \xi_1 }.
                \end{align*}
                \item[$(c)$] 
                    If $dn \equiv 0 \pmod{k}$, then
                    $ \displaystyle
                \Sc(\pv,\sv) \proj (\dv) =\Sc (\pv',\sv') \proj (\dv)\rel{ \xi_1 + \mu \xi}.$
                \end{itemize}
                  \end{lemma}
                \begin{proof} 
                    The assumptions on $\pv$ and $\sv$ and~\eqref{eq:badop} imply that for $a, b \in V$,
                    \begin{align*}
                        \bade(\pv; a, b, \dv-\eb) &\leq \frac{1}{\mu \binom{n-1}{k-1}} \cdot \binom{n-2}{k-2} \cdot 10 \mu = 10 \alpha  = o(1) \text{ \quad and \quad }\\
                \badp(\sv; a, b, \dv -\eb) &\leq   \frac{1}{\mu \binom{n-1}{k-1}} \cdot \binom{n-2}{k-1} \cdot 10 \mu^2 <
                    10 \mu < \frac 14.
                    \end{align*}
               The same estimates hold for $\pv'$ and $\sv'$. Since $\frac{d}{d_a}< 2$, ~\eqref{eq:ratop} implies that the absolute errors in $\bade$ and $\badp$ give  rise  to a relative error in $\Rc_{ab}$ of the same order of magnitude, and part (a) follows upon taking account of the multiplicative factors $O(\alpha)$ and $O(\mu)$. 
                
                For (b), the hypothesis on $\rv$ implies that $\rv_{\ub, \ua}(\dw) = \rv'_{\ub, \ua}(\dw)\rel{k (\mu + \alpha )\xi}$ for all $\ua, \ub \in \binom{V}{k-1}$ and $\dw \in Q_1^1(\dv)$.  Since $\pv \proj (\dw) \le 10\mu\le  \frac14$ for $\dw \in Q_1^0(\dv)$, the equation for $\Pc(\pv,\rv)$ follows from~\eqref{eq:probop}. 
                
                Part (c) follows immediately from~\eqref{eq:pathop}
                using a similar argument.
                \end{proof}

        	Fix $\Omega{(0)}\se \Z^n $  for the following definitions. For $s \geq 1$, let $ \Omega{(s)}$ denote the set of all $\dv\in \Omega{(0)} $ with $\sum_{v \in V} d_v \equiv 0 \pmod{k}$  for which  $Q_s^0(\dv) \se \Omega{(0)}$. %
        
        \begin{cor}\thlab{c:contraction}
        Let $k(\mu + \alpha)< c$ for $c$ sufficiently small.	Fix $\Omega{(0)}\se \Z^n $, $s \in \Z$ and $\xi <1$. 
        Let $\pv \proj (\dv),\pv'\proj (\dv) \leq 10 \mu$ and $ \sv\proj (\dv), \sv'\proj (\dv) \leq 10\mu^2$ for $\dv \in \Omega(s)$. If $(\pv , \sv )\proj(\dv) =  (\pv' , \sv' )\proj (\dv) (1 \pm \xi)$ for  $\dv \in \Omega(s)$, then $\cC(\pv, \sv)\proj (\dv) = \cC(\pv', \sv') \proj (\dv) (1 \pm e^{-1} \xi)$ for $\dv \in \Omega(s+3)$.
        \end{cor}
        \proof
        Let $\dv \in \Omega(s+2)$, $\dv' \in Q^1_1(\dv)$ and $\dv'' \in Q^0_1(\dv') \subset \Omega(s)$. We have $\pv \proj (\dv'')=\pv' \proj (\dv'')(1 \pm \xi)$ and  $\sv \proj (\dv'')=\sv' \proj (\dv'')(1 \pm \xi)$, so~\thref{l:errorImplication}(a) implies 
        $$\Rc(\pv, \sv) \proj (\dv') = \Rc(\pv', \sv') \proj (\dv') \rel{(\mu+\alpha)\xi}.$$
           Furthermore, applying~\thref{l:errorImplication}(b)  with $\rv= \Rc(\pv, \sv)$, $\rv'=\Rc(\pv', \sv')$, $\xi$ replaced by $C\xi$ 
        for a suitable constant $C$ and $c< 1/C$
        gives 
        $$\Pc(\pv, \Rc(\pv, \sv))\proj (\dv) =  \Pc(\pv', \Rc(\pv', \sv'))\proj (\dv) \rel {(\mu + \alpha)k \xi }.$$
        Now, \thref{l:errorImplication}(c) with         $\xi_1 = O((\mu + \alpha)k \xi)$ implies 
         $$\Sc(\Pc(\pv,\Rc(\pv, \sv)), \sv ) \proj  (\dv)= 
         \Sc(\Pc(\pv',\Rc(\pv', \sv')), \sv' ) \proj (\dv) \rel {(\mu + \alpha)k \xi }$$ for $\dv \in \Omega(s+3)$. Taking $c$ sufficiently small so that the error term is at most $e^{-1}\xi$, we get
         $\cC(\pv, \sv)\proj (\dv) =  \cC(\pv', \sv') \proj (\dv)(1 \pm e^{-1}\xi) $ . 
        \qed

        \section{Dense hypergraphs} \label{sec:densecase}
                We start by proving approximations to the ratios $R_{ab}(\dv)$, as well as to the edge and path probabilities in a random $k$-graph with degree sequence $\dv$. The crucial property of our approximation formulae is that they are `close to invariant' under the operator $\cC$ from Section~\ref{sec:operators}. These formulae were initially generated by starting with a crude approximation $\pv_K(\dv) = \mu$,  applying the operator $\cC$ repeatedly $j$ times, and observing the pattern in the terms remaining when those of  size $O(\mu^j)$ were dropped.
        In {this section}, it is convenient to write expressions in terms of $q=k-1$  and to  let    $\alpha = \frac{k-1}{n-1} = \frac{q}{n-1}$.  

 Recall that the average of a sequence $\dv$ is $d = \frac 1n \sum_{v \in V} d_v$ and its spread is $\max_{v \in V} |d_v-d|$.
        	  
        	\begin{lemma}   \thlab{l:ratiosDense} 
        	   There exists a positive constant $c<1$ such that the following holds. Let $n$ and $q =k-1 \geq 2$ be integers, $\eps, d_0 \in \R $, $\alpha = \frac{q}{n-1}$ and $\mu_0 = d_0\binom{n-1}{q}^{-1}$, noting that all these parameters may be functions of $n$. Assume that  $q(\mu_0 + \alpha) <c$, $\eps q  <  c^{1/3} $, $d_0 > \log^3 n$, $\eps^2 d_0 >1$ and $d_0n \equiv 1 \pmod{q+1}$. If $\dv_0 = (d_v)_{v \in V}$ is a sequence with average $d_0$ and spread at most $\eps d_0$, then {uniformly}
        	   \bee
        		R_{ab}(\dv_0) = \frac{d_a \left(\binom{n-1}{q} - d_b \right)}{d_b \left(\binom{n-1}{q} - d_a \right)}
        								\left(1 + \frac{(d_a-d_b)q}{d_0(1-\mu_0)(n-1-q)} \right) \rel{\zeta} 
                \label{eq:ratiostatement}
        		\ee
        		for $a, b \in V$, with 
        			\bee
        			    \zeta = \begin{cases}
                            \mu_0 \eps^{3} + \alpha \eps^{2} + \alpha^2 \eps + \alpha^3,     & q = 2 \\
        			            q^2\left(\mu_0 + \alpha \right) \eps^2
        			            , & q \geq 3.
        			        \end{cases}
                    \ee   
        	\end{lemma}
        	Note that the hypotheses imply $\zeta <c$ for large $n$. Before proving the  lemma, we state some simple identities used therein to aid readability.
        	
        	        \begin{claim} \thlab{lemma:sum}
                    Let $q, n \in \Z$, $\eps \in \R$, $\alpha = \frac{q}{n-1}$ and $a, b \in V$. For $v \in V$, let $\eps_v \in [-\eps, \eps]$ and assume that  $\sum_{v \in V}\eps_v=0$. Given $K\subset V$, let $\eps_K = \sum_{v \in K}\eps_v$. For an arbitrary coefficient $c$,
                    \begin{align*}
                        \binom{n-1}{q}^{-1} \sum_{\uh \in \binom{V \setminus \{a \}}{q}}  \left(c+ \eps_{\uh} \right) & = c  -\alpha \eps_a,  \\
                       \binom{n-1}{q}^{-1} \sum_{\uh \in  \binom{V \setminus \{a, b \}}{q}}  \left(c+ \eps_{\uh} \right) & =(1-\alpha) c -(\eps_a + \eps_b)\cdot \frac{\alpha(1-\alpha)}{1-\ww}, \\
                         \binom{n-1}{q}^{-1} \sum_{\uh \in  \binom{V \setminus \{a, b \}}{q-1 }}  \left(c+ \eps_{\uh} \right) & =\alpha c - (\eps_a + \eps_b) \cdot \frac{\alpha(\alpha-\ww)}{1- \ww}.
                    \end{align*}
                    Moreover, for $c_{00}, c_{01}, c_{02}, c_{11} \in \R$,
                    \begin{align*}
                       & \binom{n-1}{2}^{-1} \sum_{ \{ u,v\} \in  \binom{V \setminus \{a, b \}}{2}} c_{00} + c_{01} ( \eps_u + \eps_v) + c_{02} (\eps_u^2 + \eps_v^2 ) + c_{11} \eps_u \eps_v \\ 
                    = &(1-\alpha)c_{00} - \alpha c_{01}(\eps_a + \eps_b)+ \alpha{\sum_{v \in V}} \eps_v^2 + (c_{01} + c_{02} +  c_{11}) O(\eps \alpha^2 +\eps^2 \alpha).\qedhere
                    \end{align*} %
        \end{claim}
    	 The proof is omitted because the identities follow from simple algebraic manipulations --- for instance, for the first equation, note that each vertex $v\neq a$ appears in $\binom{n-2}{q-1}$ of the summands and use the hypothesis that $\displaystyle \sum_{v\in V} \eps_v = 0$.    	
    
        	\begin{proof}[Proof of Lemma~\ref{l:ratiosDense}]

        	    We will approximate the functions $P$, $R$ and $\sp$ defined in Section~\ref{sec:recursions}. 
        	   For a vertex $v$ and sequence $\dv \in \Z^n$ with average $d$, we define $ \mu = \mu(\dv) = d \binom{n-1}{q}^{-1}$ and  $\eps_v = (d_v-d)/d$. Note that $$\sum_{v \in V} \eps_v =0 $$ by definition. As before, let $\eps_K = \sum_{v \in K} \eps_v$ for $K \subset V$.  For $q \geq 3$, $a, b \in V$, $\uh \in \binom{V}{q+1}$ and $\uj \in \binom{V\setminus \{a, b \}}{q}$, set
                        \begin{align}   
                        \Pf_{\uh}(\dv) %
                        &=\mu \left(1+  \frac{ \eps_{\uh}}{1-\alpha} \right), \non
                        \\
                        \Rf_{ab} (\dv) 
                        &=\frac{(1-\mu)(1-\alpha)+ \eps_a}{(1-\mu)(1-\alpha)+\eps_b}, \label{eq:rf}
                        \\
                        \Sf_{a,\uj,b} (\dv)&=  \mu^2\left(1 + \frac{\eps_a + \eps_b + 2\eps_\uj}{1- \alpha} \right) . \non
                        \end{align}
                        
        	    These will turn out to approximate $P, R, \sp$ sufficiently well when $q \geq 3$. For $q=2$ (corresponding to 3-uniform hypergraphs), we need more precise approximation formulae. As usual, the dependence on $q$ is suppressed in the notation. For distinct vertices $a, b, v, u$ in $V$, define
                    \begin{align}
                        \Pft_{auv}(\dv)&= \mu 
                    \left(1 + \frac{\eps_a}{1-\alpha} \right) \left(1 + \frac{\eps_u}{1-\alpha} \right) \left(1 +   \frac{\eps_v}{1-\alpha} \right)
                    \left(1 - \frac{\mu(\eps_a \eps_u + \eps_u \eps_v + \eps_v \eps_a)}{1- \mu} \right),
                        \non \\
                        \Rft_{ab} (\dv) &= \frac{(1-\mu)(1-\alpha) +\eps_a -\mu(\eps_a+\eps_b)}{(1-\mu)(1-\alpha) +\eps_b -\mu(\eps_a+\eps_b)}  , \label{eq:rf3} \\
                         \Sft_{auvb} (\dv) &= 
                         \mu^2\Big(1-2(1-\mu){ d}^{-1}+(1+\alpha)(\eps_a + \eps_b + 2\eps_u + 2\eps_v) + \eps_a\eps_b + \eps_u^2 + \eps_v^2   \Big)  
                             \non
                             \\
                         &\qquad + \frac{\mu^2(2-3\mu)}{1-3\mu}\Big(\eps_a \eps_u + \eps_a \eps_v + \eps_b \eps_u + \eps_b \eps_v +2\eps_u \eps_v + \eps_u^2 + \eps_v^2 \Big) .
                         \non
                    \end{align}
        	     We next informally outline the argument we give below that the functions approximate $P, R$ and $\sp$. Given $q$, $\eps$ and $d_0$ as in the statement, let ${\D_0}$ be the set of sequences with average $d_0$ and spread at most $\eps d_0$.  We start by showing that $\cC^{t_0}(\Pf, \Sf)$ is close to $(\Pf, \Sf)$ in a neighbourhood of ${\D_0}$, where $t_0 = 2 \log (nd_0)$. Secondly, from the recursive relations in Section~\ref{sec:recursions}, we deduce that $(P, \sp)$, restricted to ${\D_0}$, is a fixed point of $\cC$.
        	   Finally, we use the contraction  property of $\cC$ as expressed in \thref{c:contraction}
                to show that $\cC^{t_0}(\Pf, \Sf)$ and $\cC^{t_0}(P, \sp)$ are at distance roughly $e^{-t_0}$ in the relevant metric. \thref{l:ratiosDense} then follows from these three statements and the triangle inequality. The computations below split into two cases depending on $q$ since we need more precise estimates for $\Pft, \Sft$ and $\cC(\Pft, \Sft)$ when $q=2$.
        
                 Let $\Omega{(0)}$ be the set of sequences $\dv  \in \Z^n$ that are at $\Lh$-distance at most $3t_0 +3 = 6 \log (nd_0) + 3$   from a sequence in ${\D_0}$. As in~\thref{c:contraction}, for $s \geq 1$, let $ \Omega{(s)}$ denote the set of all $\dv\in \Omega{(0)} $ with $\sum_{v \in V} d_v \equiv 0 \pmod{q+1}$  for which  $Q_s^0(\dv) \se \Omega{(0)}$.  In particular, $\Omega(3t_0 +3)$ contains ${\D_0}$.%
                
                We remark that any $\dv \in \Omega(0)$ still  has spread at most $2\eps d_0$ for $n$ large enough. Namely, denoting the average of $\dv$ by $d \sim d_0$ and using the hypotheses $d_0> \log^3 n$ and $\eps^2 d_0>1$, we have $\log n / d_0 < d_0^{-2/3} < \eps^{4/3}$ and thus 
                    $$\left | \frac{d_v}{d} - 1 \right | \leq \eps + O\left(\frac{\log(nd_0)}{d_0} \right) = (1 + o(1)) \eps$$ for all $v \in V$.
                    
                    The following lemma states that the functions $\Pf, \Rf, \Sf$ on $\Omega(0)$ are \emph{approximate invariants} of the operators. No divisibility or graphicality assumptions are required since this is an analytic statement. 
                \begin{claim}   \thlab{claim:invariant}
                   Under the assumptions of the present lemma, for $q \geq 2$ and $\dv \in \Omega(0)$, we have {uniformly}
                    \begin{enumerate}[(a)]
                        \item $\Rc (\Pf, \Sf) {\proj} (\dv) = \Rf {\proj} (\dv) \rel{\zeta} $,
                        \item 
                    $\Pc (\Pf, \Rf) {\proj} (\dv) = \Pf {\proj} (\dv) (1 + O(\vartheta))$ and
                    \item
                    $ \Sc(\Pf, \Sf) {\proj} (\dv) = \Sf  {\proj}(\dv)\rel{\vartheta}$,
                    \end{enumerate} 
                    where  
    				$$ \vartheta = \begin{cases}
                     \eps^3 + \alpha \eps^2 + \alpha^2 \eps + \alpha^3, &q=2\\
                      q^2\eps^2, & q \geq 3.  
                     \end{cases}$$
                    Consequently, $\cC(\Pf, \Sf) \proj (\dv) = (\Pf, \Sf)\proj (\dv) \rel{\vartheta}$ {uniformly} for all $\dv \in \Omega(3)$.
                \end{claim}
                    \begin{proof}[Proof of \thref{claim:invariant}]
                      We start by establishing how small changes in the sequence $\dv$ affect the function $\Pf\proj(\dv) $. Since the average  {of the individual components} of any sequence in $\Omega(0)$ is asymptotically equal to $d_0$, we will be systematically replacing the error term $O\left(d^{-1} \right)$ by $O \left(d_0^{-1}\right)$.  {Moreover, the assumptions of the lemma imply that when $c$ is sufficiently small, we have $\vartheta < c^{2/3}$  and, as noted above, $\zeta < c$.}
                      Consider a sequence $\dv \in \Omega(0)$ and $\dv' = \dv -  \ve_{\uj}$, where $\uj \subset V$ with $|\uj| = q+1$. Let $d$ and $d' $ be the average of $\dv$ and $\dv'$, respectively. For any vertex $u$, let $\mu'$ and $\eps_u '$ be the parameters corresponding to $\mu $ and $\eps_u$ with respect to the sequence $\dv'$. This will be the convention throughout the proof. We have $d' = d - \frac{q+1}{n} = d \rel{\alpha/d_0}$ and hence                   \bee \label{eq:changedvmu}
                        \mu' = \mu \rel{\frac{\alpha}{d_0}}
                      \text{\quad and \quad} \eps'_{u} = \eps_u + O \left( \frac{\alpha}{d_0}\right) \text{ for }  u \notin J
                      \ee
                       Moreover, for any $u \in \uj$,
                        \bel{eq:changedveps}
                            \eps_u' =\frac{d_u -1}{d\relf{q}{d_0n}} -1 
                            = \left(\eps_u - \frac{1}{d} \right) + O\left(\frac{q}{d_0n} \right) = \eps_u + O\left(\frac{1}{d_0} \right). %
                            \non
                        \ee
                         Now we split the calculation into two cases.
                        \paragraph{Case 1:~$q\geq 3$.}
                        Recall the definitions of the functions $\Pf, \Rf, \Sf$  in~\eqref{eq:rf}, and that we write $\eps_\uh = \sum_{u \in \uh} \eps_u$.  It follows that, for $\dv' = \dv - \ve_{J}$,
                        \begin{align}
                            \label{eq:changedv}
                            \Pf_{\uj}(\dv') = \mu' \left(1 + \frac{\eps_{\uj}'}{1- \alpha} \right) 
                        &=\mu   \left( 1 + \frac{1}{1-\alpha} \left( \eps_{\uj}-\frac{q+1}{d} \right) + O \left(\frac{q^2}{d_0n} \right) \right) \\
                        &= \mu \left(1+ \frac{\eps_{\uj}}{1-\alpha} \right)  \relf{q}{d_0}. \non 
                        \end{align}
                    Thus  $\Pf_{\uj}(\dv') = \Pf_{\uj}(\dv)\relf{q}{d_0}$, which absorbs fluctuations due to perturbations of the degree sequence into the error term. Even though $d_0 \eps^2 > 1$ is assumed in the lemma, we often keep error terms involving $1/d_0$ to facilitate identifying the source of the error, namely,  from perturbing \emph{few} entries of the degree sequence $\dv$. On the other hand, $O(\eps^2)$ terms will arise from Taylor expansions.
                        
                        To show (a), observe that $\Rf_{aa}(\dv) =1 $ by definition~\eqref{eq:rf}.       
               Let $a, b \in V$ with $a \neq b$ and $\dw = \dv - \eb$. We compute $\badp$ and $\bade$ to the accuracy which is required for the ratio approximation. By definition of $\badp$ in \eqref{eq:badop}, definition of $\Sf$ in~\eqref{eq:rf} and using~\eqref{eq:changedvmu}, 
                            \begin{align*} \non
                                \badp (\Sf; a, b, \dw) &=
                                \frac {1}{\mu' \binom {n-1}{q} } \sum_{\uh \in \binom{V \setminus \{a, b \}}{q}} \Sf_{a, \uh, b } (\dw) \\
                                &= \frac {1}{\mu \binom {n-1}{q} } \sum_{\uh \in \binom{V \setminus \{a, b \}}{q}} 
                                 \mu^2\left( 1 + \frac{\eps_a + \eps_b + 2\eps_\uh}{1- \alpha} \right)\relf{1}{d_0},
                            \end{align*}
                        where the error term is due to $\eps_b' = \eps_b + O(1/d_0)$ and $
                        \eps_K' = \eps_K (1 + O(q/(dn)))$. To estimate the sum, we apply Claim~\ref{lemma:sum}. It follows that
                            \begin{align*}
                            	               \badp (\Sf; a, b, \dw)  &= \mu\left((1 - \alpha + \eps_a + \eps_b) - 2(\eps_a +\eps_b)\beta_1(\alpha, n) \right) \relf{1}{d_0} 
                            \end{align*}
                with $\beta_1(\alpha, n) \leq 2 \alpha$. Now we estimate $\bade$. Note that now $\uh$ is a $(q-1)$-tuple. Similar to the above, for a function $\beta_2(\alpha, n) \leq 2 \alpha $, we have
                \begin{align*}
                    \bade (\Pf ; a, b, \dw)   &= \frac {1}{\mu' \binom {n-1}{q} } \sum_{\uh \in \binom{V \setminus \{a, b \}}{q-1}} \Pf_{a+\uh, b}(\dw)  \\
                    &=  \frac {1}{\binom {n-1}{q} } \sum_{\uh \in \binom{ V\setminus \{a, b \}}{q-1}} \left(1 + \frac{\eps_a + \eps_b}{1-\alpha} + \frac{\eps_{\uh}}{1-\alpha} \right) \relf{1}{d_0}\\
                    &=\alpha \left(1 + \frac{\eps_a + \eps_b}{1-\alpha}- (\eps_a + \eps_b)\beta_2(\alpha, n) \right) \relf{1}{d_0}
                \end{align*}
                using Claim~\ref{lemma:sum}. It follows that there is a function ${ \beta } = \beta (\alpha, n) = O(\mu + \alpha)$ independent of $a$ and $b$ such that
                \bee 
                    \bade(\Pf; a, b, \dv - \eb) + \badp(\Sf; a, b, \dv - \eb)  = (\mu + \alpha - \mu \alpha + (\eps_a + \eps _b)\beta) \rel{d_0^{-1}},
                \ee
                and the same estimate holds for $\bade(\Pf; b, a, \dv - \ea) + \badp(\Sf; b, a, \dv - \ea)$ in the denominator in the definition of $\Rc$~\eqref{eq:ratop}. Therefore,
                $$\Rc(\Pf, \Sf)_{ab}(\dv) = \frac{1+\eps_a - \mu-\alpha  + \mu \alpha- (\eps_a+\eps_b)\beta}{1+\eps_b -\mu-\alpha  + \mu \alpha- (\eps_a+\eps_b)\beta } \relf{\mu + \alpha}{d_0}.$$
                
                At this point, we have a non-trivial cancellation --- discarding the term $(\eps_a + \eps_b)\beta$ only incurs a relative error of $O(\eps(\eps_a + \eps_b)\beta)$. Since this fact will be used again, we include the computation. We have
                    \bee \label{eq:ratiocomp}
                        \frac{(1-\mu)(1-\alpha) +\eps_a + (\eps_a + \eps_b)\beta}{(1-\mu)(1-\alpha) + \eps_b + (\eps_a + \eps_b)\beta} \cdot \frac{ (1-\mu)(1-\alpha)+\eps_b}{ (1-\mu)(1-\alpha)+\eps_a} = 
                        1+ \frac{(\eps_a-\eps_b)(\eps_a + \eps_b)\beta}{(1+O(\mu + \alpha))} = 1 + O(\eps^2 \beta).
                    \ee
                Recalling that $\eps^2 > d_0^{-1}$ and $\zeta = q^2(\alpha + \mu_0)\eps^2$, we conclude that
                    $$\Rc (\Pf, \Sf)_{ab} (\dv) = \frac{(1-\mu)(1-\alpha)+ \eps_a}{(1-\mu)(1-\alpha)+\eps_b}\rel{\eps^2 \beta + (\mu_0 + \alpha)\cdot \frac{1}{d_0} } = \Rf_{ab}(\dv)\rel{\zeta},$$
                    completing the first part of the claim.
                    
                    To show (b), we turn to estimating $\Pc(\Pf, \Rf)_{A, v} (\dv)$ according to the definition in~\eqref{eq:probop}. Let $\nu = (1-\mu)^{-1}(1-\alpha)^{-1}$ to avoid too many double fractions. Let $\uh = \ua +  v$ with $\ua= \{a_1, a_2, \dots ,  a_q \}$ indexed in increasing order. Fix a set $\ub$ and its $\ua$-consistent ordering $( b_1, b_2, \dots, b_q )$. We start by estimating $\Rfs_{\ub,\ua}(\dv - \ev)$. For $\dw =\dv  - \ve_{b_1b_2\cdots b_{j-1}a_{j+1}\cdots a_q v}$, we have  $\mu' = \mu \relf{q}{d_0n}$ and $\eps_{b_j}' = \eps_{b_j} + O \left(\frac{q}{d_0 n} \right)$ (since $d_{b_j} = d_{b_j}'$), so
                    \bee
                        \Rf_{b_j a_j} (\dv  - \ve_{b_1b_2\cdots b_{j-1}a_{j+1}\cdots a_q v})=
                        \frac{(1-\mu)(1-\alpha)+ {\eps_{b_j}  + O\left( \frac{q}{d_0n} \right) } } { (1-\mu)(1-\alpha) + {\eps_{a_j} + O\left( \frac{q}{d_0 n} \right)  } } = 
                        \frac{1 + \nu \eps_{b_j} }{1 + \nu \eps_{a_j} }\relf{ q}{d_0n}.
                        \label{eq:changedv_rat}
                    \ee
                    Therefore, using~\eqref{eq:changedv} and linearising in $\eps$-terms, we get 
                    \begin{align*}
                        &\frac{1- \Pf_{\ub,v }(\dv - \ve_{B, v})}{1- \Pf_{\ua,v }(\dv - \ve_{A, v})} \cdot \Rfs_{\ub, \ua}(\dv - \ev) \\ 
                        &= \frac{1- \mu - \mu(1-\mu)\nu(\eps_{\ub}+ \eps_{v})}{1- \mu - \mu(1-\mu)\nu(\eps_{\ua}+ \eps_{v})} \cdot \relf{ \mu_0 q}{d_0} \prod_{j=1}^q  \frac{1+\nu \eps_{b_j} }{1 + \nu \eps_{a_j}} \relf{  q}{d_0n}   \\
                        &=\frac{1-\mu\nu \eps_{\ub}} {1-\mu\nu\eps_{\ua}} \cdot \frac{1+ \nu \eps_{\ub}}{ 1 + \nu \eps_{\ua}} \rel{  q^2 \eps^2 } \\
                        &= \frac{1 + (1-\mu)\nu \eps_{\ub}}{1+(1-\mu)\nu \eps_{\ua}}\rel{  q^2 \eps^2},
                        \end{align*}
                       where the error term $O(q^2 \eps^2)$ comes from linearising the product, and recalling that $\eps^2 >d_0^{-1}$.
                    Summing over $\ub$ as before, we get
                    $$\binom{n-1}{q}^{-1} \sum_{\ub \in \binom{V\setminus \{v \}}{q}} \frac{1- \Pf_{\ub}(\dv - \ve_{B,v})}{1- \Pf_{\ua}(\dv - \ve_{A,v})} \cdot \Rfs_{\ub, \ua} (\dv - \ev) 
                    = \frac{1+ O(q^2 \eps^2)}{1+(1-\mu)\nu \eps_{\ua}}\cdot \left(1 - \alpha (1-\mu)\nu \eps_v \right).$$
                    Finally,
                    $$\Pc(\Pf, \Rf)_{\ua, v}(\dv) = \mu(1+\eps_v)\cdot  \frac{1+(1-\mu)\nu \eps_{\ua}}{1-\alpha(1-\mu)\nu \eps_v} \rel{q^2 \eps^2}.$$
                    The coefficient of $\eps_v$ in the Taylor expansion is $1+ \alpha (1- \mu) \nu = 1 + \frac{\alpha}{1-\alpha} = \frac{1}{1-\alpha}$. Therefore
                    $\Pc (\Pf, \Rf)_{\ua, v}(\dv) = \mu\left(1 + \frac{\eps_{\ua} + \eps_v}{1- \alpha} \right)\rel{q^2 \eps^2},$
                    which establishes (b).
                    
                	To estimate $\Sc(\Pf, \Sf)$, we use~\eqref{eq:pathop} and, once again,~\eqref{eq:changedv} to obtain
        			\begin{align*}
        					\Sc(\Pf, \Sf)_{a,\uh,b} (\dv) &=\mu^2 \cdot \frac{1+ \frac{\eps_a + \eps_\uh}{1-\alpha}}{1-\mu - \frac{\mu(\eps_a + \eps_\uh)}{1-\alpha} } \left(1 + \frac{\eps_b + \eps_\uh}{1-\alpha}- \mu - \frac{\mu(\eps_a + \eps_b + 2 \eps_\uh)}{1-\alpha}  \right) \relf{q}{d_0} .
        			\end{align*}
        			Each of the variables $\eps_a, \eps_b$ and $\eps_\uh$ has Taylor coefficient $\frac{\mu^2}{1-\alpha}$, and hence $$\Sc(\Pf, \Sf)_{a,\uh,b} (\dv) =\mu^2 \left( 1 + \frac{\eps_a + \eps_b +\eps_{\uh}}{1 - \alpha }\right)\rel{q^2 \eps^2}.$$ This completes the proof of (c).
        			
        			To justify the final statement on the operator $\Cc$, we need to combine the above estimates. Let $\dv \in \Omega(3)$. For $\Pc$, we have $$\Pc(\Pf, \Rc(\Pf, \Sf)) \proj (\dv) = \Pc (\Pf, \Rf) \proj  (\dv) \rel{ \eps^2} = \Pf \proj (\dv) \rel{q^2\eps^2}, $$ 
        			where the first equation follows from~\thref{l:errorImplication} (b) and  $\Rc(\Pf, \Sf)) \proj (\dv') = \Rf \proj (\dv') \rel{\eps^2} )$ for $\dv' \in Q_1^1(\dv')$.
        			Similarly, using~\thref{l:errorImplication} (c),
        			$$\Sc(\Pc(\Pf,\Rc(\Pf, \Sf)), \Sf ) {\proj}(\dv) = \Sc(\Pf, \Sf) {\proj}(\dv)\rel{q^2 \eps^2} = \Sf {\proj}(\dv)\rel{q^2 \eps^2}.
        			$$
        			By definition of $\Cc$, these two estimates give
        			    $$\cC(\Pf, \Sf) \proj (\dv) = (\Pf, \Sf)\proj (\dv) \rel{\vartheta},$$
        			    as required.
        			
        			\paragraph{Case 2.} If $q=2$, the functions $\Pf, \Rf, \Sf$ are defined in~\eqref{eq:rf3}. %
        			 We follow the same outline for the computations. 
        			 As before, let $\dv \in \Omega(0)$. We start by comparing $\Pft(\dv)$ with $\Pft(\dw)$ for a `perturbed' degree sequence $\dw = \dv - \ve_{\uh}$, where $\uh \subset V$ with $|\uh|\leq 10$.  { For $x \in \uh$, we have $\eps_x' = \eps_x - \frac1d +  O \left(1/(d_0n) \right)$, but those perturbations are only significant in terms which are linear in $\eps$. Moreover, using~\eqref{eq:changedvmu}, the error in replacing $\mu'$ and $\eps'_x$ by $\mu$ and $\eps_x$ for $x \notin K$ is negligible. Hence}
                      \begin{align}
                            \Pft_{auv}(\dv- \ve_{K})  &= \mu \left(1-\frac{\langle \ve_{\uh},  \ve_{auv} \rangle}{d} \right)
                        \left(1 + \frac{\eps_a}{1-\alpha} \right) \left(1 + \frac{\eps_u}{1-\alpha} \right) \left(1 +   \frac{\eps_v}{1-\alpha} \right)  \non \\
                        &\quad\times\left(1 - \frac{\mu(\eps_a \eps_u + \eps_u \eps_v + \eps_v \eps_a)}{1- \mu} \right)\relf{\alpha + \eps}{d_0} 
                        \non \\ 
                        & = \Pft_{auv}(\dv) \left(1-\frac{\langle \ve_{\uh}, \ve_{auv} \rangle}{d} \right) \rel{   \vartheta   },      
                        \label{eq:changedv_prob} 
                    \end{align}
                    since each vertex in $K \cap \{a, u, v \}$ contributes a factor of $(1-1/d+O((\alpha+\eps)/d_0))$.
                     
                     Similarly,  {in $\Sft_{auvb}(\dv - \ve_{\uh})$, the significant contributions correspond  to the vertices in $K \cap \{a, u, v\}$ and $K \cap \{u, v, b \}$, so}
                    \begin{align}    
        			\Sft_{auvb}(\dv - \ve_{\uh}) =& \mu^2 \left(1- \frac{2-2\mu}{d}+(1+\alpha)\left(\eps_a + \eps_b + 2\eps_u + 2\eps_v -\frac{\langle \ve_{\uh}, \ve_{auv}+ \ve_{uvb} \rangle }{d} \right) + \eps_a\eps_b + \eps_u^2 + \eps_v^2    \right) \non
        			\\
                              &  + \frac{\mu^2(2-3\mu)}{1-3\mu} \left( \eps_a \eps_u + \eps_a \eps_v + \eps_b \eps_u + \eps_b \eps_v +2\eps_u \eps_v + \eps_u^2 + \eps_v^2 \right)  \rel{  \frac{\alpha + \eps}{d_0}} .  \label{eq:changedv_path}
                    \end{align}
             For (a), observe that $\Rft_{aa}(\dv) =1$ by definition~\eqref{eq:ratop}. 
                Let $a, b \in V$ with $a\neq b$ and $\dv' = \dv - \eb$. The computation of  $\Rc_{ab} (\dv)$ is simplified by exploiting the symmetry between $a$ and $b$ in the operator definition. By definition of $\bade$, using~{\eqref{eq:changedv_prob}} and~\thref{lemma:sum}, we have
                \begin{align*}
                    \bade (\Pft; a, b, \dw)  &= \frac {1}{\mu \binom {n-1}{2} } \sum_{u \in V \setminus \{a, b \} } \Pft_{aub}(\dw)   = \frac{1}{\binom {n-1}{2}} \sum_{u \in V \setminus \{a, b \} } \left( 1 + \frac{\eps_a + \eps_u + \eps_b}{1-\alpha} \right) \rel{\eps^2 + \frac{1}{d_0}}\\
                    & = \alpha(1 + c_1 {(\eps_a,\eps_b)}  +  O( \eps^2)),
                \end{align*}
                    where $c_1(\eps_a, \eps_b) = O(\eps)$ is symmetric in $a$ and $b$.
                     Similarly, using~\eqref{eq:changedv_path} and~\thref{lemma:sum},
                \begin{align*}
                    \badp (\Sft; a, b, \dw) &= \frac {1}{\mu \binom {n-1}{2} } \sum_{  { \{u, v\} \in \binom{V \setminus \{a, b \}}{2} } } \Sft_{auvb}(\dw)  \\
                    & = \mu(1 -\alpha + \eps_a + \eps_b + c_2 {(\eps_a,\eps_b)} +O( \alpha^2 \eps  + \alpha^3 + \eps^3)),
                \end{align*}
                where $c_2 {(\eps_a, \eps_b)} = O(\alpha \eps + \eps^2)$ is symmetric in $a$ and $b$ (since  $\Sft_{auvb}(\dw)$ is also symmetric up to $O(\eps^3)$ terms). 
                
                Now we are ready to compute $\Rc(\Pft, \Sft)_{ab}(\dvB{})$.   By definition of $\Rc$ and the above computations, we have
                $$\Rc(\Pft, \Sft)_{ab}(\dv) = \frac{1+\eps_a -\alpha - \mu(1-\alpha)- \mu(\eps_a+\eps_b) - \alpha c_1 - \mu c_2 }{1+\eps_b -\alpha - \mu(1-\alpha) - \mu(\eps_a+\eps_b) - \alpha c_1 - \mu c_2  } \rel{\zeta}.$$
                Using~\eqref{eq:ratiocomp} to eliminate $c_1$ and $c_2$, we conclude that
                    $$\Rc (\Pf, \Sf)_{ab} (\dv) = \frac{1+\eps_a - \alpha - \mu(1 - \alpha +\eps_a+\eps_b)}{1+\eps_b - \alpha - \mu(1 - \alpha +\eps_a+\eps_b)} \rel{\zeta + \alpha \eps c_1 + \mu \eps c_2 } = \Rft_{ab}(\dv)\rel{(\zeta},$$
                    completing part (a) for $q=2$.

                To show (b), we will compute $\Pc(\Pft, \Rft)_{auv}(\dv)$ according to the formula~\eqref{eq:probop}. Recall that, as in Case 1, we can replace the relevant approximations $\Rft$ and $\Pft$ by the corresponding quantities evaluated at $\dv$ using~\eqref{eq:changedv_rat} and~\eqref{eq:changedv_prob}. Therefore, we have
                	\begin{equation} \label{eq:prf3}
                	\Pc(\Pft, \Rft)_{auv}(\dv) =d_{v} \left(\sum_{bw \in X}  K_{bw} \right)^{-1} \rel{\frac{1}{d_0n} + \eps^3},
                	\end{equation}
                	where $X=\{b, w \in V  \setminus \{v \}: b< w \}$ and
                \bee
                    K_{bw} = \Rft_{ba} (\dv)\Rft_{wu} (\dv)
                	\frac{1-\Pft_{bw v}(\dv ) \left(1 - 3d^{-1} \right)}
                	{1-\Pft_{au v}(\dv)\left(1-3d^{-1} \right)}
                \ee
                is a rational function in $\eps_b$ and $\eps_w$. Expanding about $\eps_b$ and $\eps_w$, one can see that $K_{bw} = \kappa_0 + \eps_b + \eps_w + \kappa_1(\eps_b + \eps_w)+ O(\eps^3)$
                { for some functions $\kappa_0$ and $\kappa_1$ which are independent of $\eps_b$ and $\eps_w$ and satisfy} 
                $\kappa_1 = O(\alpha + \eps)$. The coefficient of $\eps_b$ can be computed using computer algebra, or by inspection -- the only significant terms are $(1-\mu)\eps_b$ coming from $\Rft_{ba}(\dv)$ and $\mu \eps_b$ coming from $\Pft_{bwv}(\dv)$. 
                Similarly, $ \kappa_0 = \frac{1}{(1+ \eps_a)(1+\eps_u)}+(\eps_a \eps_u + \eps_u\eps_v + \eps_v \eps_a)\cdot \frac{\mu}{1-\mu} -\alpha(\eps_a + \eps_u)  + O(\vartheta)$, with the first term coming from $\Rft_{ba} (\dv)\Rft_{wu} (\dv)$. Summing over $b$ and $w$, and using $\sum_{b \in V} \eps_b =0$, we get
                \begin{align*}\sum_{bw \in X} K_{bw} &= \kappa_0 - \alpha \eps_v + O(\alpha^2 \eps + \alpha \eps^2 + \eps^3) \\
                &= \frac{1}{(1+ \eps_a)(1+\eps_u)}+(\eps_a \eps_u + \eps_u\eps_v + \eps_v \eps_a)\cdot \frac{\mu}{1-\mu} -\alpha(\eps_a + \eps_u) - \alpha  \eps_v + O(\vartheta). 
                \end{align*} 
                
                Now we can check that the right-hand side of~\eqref{eq:prf3} coincides with $\Pft_{auv}(\dv)$ given in~\eqref{eq:rf3} up to $O(\vartheta)$. %
                View both expressions as rational functions in $\eps_a, \eps_u$ and $\eps_v$. Now, the coefficient of $\eps_a, \eps_u$ and $\eps_v$ in both expressions is $\mu(1+ \alpha) + O(\alpha^2)$. Moreover, the coefficient of $\eps_a \eps_u$ is $\frac{-\mu}{1-\mu} + O(\alpha)$. It follows that $\Pc(\Pft, \Rft)_{auv}(\dv)  = \Pft_{auv}(\dv)\rel{\vartheta}$, as required for (b).

                 For (c), by definition of $\Sc$~\eqref{eq:pathop} and approximations~\eqref{eq:changedv_prob} and~\eqref{eq:changedv_path},
                            \bee    
                                \Sc(\Pft, \Sft)_{auv b}  (\dv)= 
                                \frac{\Pft_{auv}(\dv)}{ 1- \Pft_{auv}(\dv  ) \left( 1- \frac 3d  \right) } \left(\Pft_{b uv} (\dv ) \left( 1- \frac 2d \right) - \Sft_{auv b}(\dv)\left( 1-\frac 5d  \right) \right)\rel{\vartheta}.
                                \non
                            \ee
                    One can check that up to the desired accuracy, the right-hand side agrees with $\Sft_{auvb}(\dv)$. {To assist with this, since the negligible terms $\eps^3$, $\alpha \eps^2$, $\alpha^2 \eps$ and $\alpha^3$ have total degree 3 in  $\alpha$ and $\eps$, we may introduce a size variable $\sf y$, substitute $\sf \alpha:= \alpha* y$, $\sf \eps := \eps * y$, $d:= d / {\sf y}^2$ and discard all $O(\sf y^3)$-terms.}  This completes the proof of (c).

                As in Case 1, the statement for $\cC$ follows from the estimates above and \thref{l:errorImplication}.  This finishes the proof of \thref{claim:invariant}.
        	\end{proof}
       
        The remainder of the proof of \thref{l:ratiosDense} is unified, i.e.~it holds for $q \geq 2$ satisfying the hypotheses of the present lemma. 
            We have just shown that $\Cc(\Pf, \Sf) \proj (\dv) = (\Pf, \Sf) \proj (\dv)(1 \pm C_1\vartheta)$ for an absolute constant $C_1$ and $\dv \in \Omega(3)$. Next, we
                 claim that for $\dv \in \Omega(3t_0)$, $\cC^{(t_0)}(\Pf, \Sf) \proj(\dv)= (\Pf, \Sf) \proj (\dv)\rel { C_1 \vartheta}$. 
                 
                 To see this, note that iterating \thref{c:contraction}  gives $\cC^{j}(\Pf, \Sf) \proj (\dv) = \cC^{j-1}(\Pf, \Sf) \left( 1 \pm e^{-j} C_1\vartheta \right)$ for $\dv \in \Omega(3j)$. Recall that $\Omega(3t_0) \subset \Omega(j)$ for $j \leq 3t_0$. Therefore, if $\dv \in \Omega(3t_0)$, the previous inequality implies that
                \begin{align}   \label{eq:cinv}
                    \cC^{(t_0)}(\Pf, \Sf)\proj (\dv) = (\Pf, \Sf) \proj (\dv) \left (1 \pm  \left( \sum_{j = 0}^\infty e^{-j} \right)  C_1 \vartheta \right)   = (\Pf, \Sf) \proj (\dv)  \rel{ \vartheta }. 
                 \end{align}

        		The invariance of $(P, \sp)$ under $\cC$ follows from~\thref{t:recrel}, but we need to check that the $k$-graphicality 
        		 conditions of the lemma are satisfied.        As usual, all the inequalities hold for $n$ sufficiently large.
                We established that any $\dv \in \Omega{(0)}$ has spread at most $\eps d_0 = o(d/q)$, so Proposition~\ref{prop:kgraphical}~(i) implies that $\Num(\dv) > 0$ provided $\sum_{v \in V}d_v \equiv 0 \pmod{q+1}.$ 
                Now consider any $\uh \in \binom{V}{q+1}$ and  $\dv\in \Omega{(3)}$.~\thref{l:simpleSwitching}, together with $\mu<c/q$ and $d_v \leq d(1 + o(1/q))$ implies that for $\uh \in \binom{V}{q+1}$ and $a \in V$,
                \bel{Pbound2}
                 P_{\uh}(\dv) \leq 
         			 \frac{2 q!}{(m(q+1))^q} \cdot  \prod_{v \in \uh} d_v  \leq \frac{3q! d}{n^q}
                \quad  \mbox{for all $\dv\in \Omega{(0)}$}.
                \ee 
                 In particular, $\Num_{\uh} (\dv) < \Num (\dv)$. Moreover, for $\dv \in \Omega{(1)}$, \eqref{Pbound2} implies that  $ \Num(\dv-\ve_{\uh})>0$ and $P_{\uh}(\dv-\ve_{\uh}) <1$.
                 Thus $ \Num_{\uh}(\dv)>0$ by Lemma~\ref{trick17}. Therefore,~\thref{t:recrel} implies that for any $\dv \in \Omega {(3)}$,      \bel{eq:cinv1} \cC(P, S) \proj (\dv) =(P, S ) \proj(\dv).\ee 

                Now we move on to comparing $\cC^{t_0}(P, \sp)$ and $\cC^{t_0}(\Pf, \Sf)$. %
                 We start by establishing some initial bounds on $P$ and $\sp$ so that the conditions of~\thref{c:contraction} are satisfied.
                We have shown that 
                for $\dv \in \Omega{(0)}$ with $\sum_{v \in V}d_v \equiv 0 \pmod{q+1}$,  %
                $P \proj (\dv)<\frac{3q!d}{n^q} \sim 3\mu_0$, where the approximation $\binom nq \sim n^q / q!$ is valid since $q^2  = o(n)$. Relation~\eqref{eq:Sfixed} then gives $\sp_{a, \uh, b} (\dv) \leq 2 P_{a, \uh}(\dv) P_{b, \uh} (\dv-\ve_{a, \uh}) 
                \leq 20 \mu_0 ^2 $ for $\dv \in \Omega(1)$ and sufficiently large $n$. By definition of $\Rc$, it follows that $R {\proj}(\dv- \ve_\uh) =\Rc(P, \sp) {\proj} (\dv - \ve_{\uh}) = 1 + O(\eps + \alpha + \mu_0) = 1+O(c)$ for $\uh \in \binom{V}{k-1}$ and $\dv \in \Omega(2)$. Applying $\Pc$, we infer $P {\proj}(\dv) = \Pc(P, R) {\proj}(\dv) = \mu_0 \rel{c}$.
                We conclude that  $P \proj (\dv) =  \mu_0\rel{c} = \Pf \proj (\dv) \rel{c}$ for $\dv \in \Omega(2)$.
                
                Furthermore, since $\sp \proj (\dv) \leq 20 \mu_0^2 \leq 40 \Sf \proj (\dv)$ for $\dv \in \Omega(0)$ with $\sum_{v \in V}d_v \equiv 0 \pmod{q+1}$, we can apply~\thref{l:errorImplication} (c) with $\xi_1 = O(c)$ and $\xi = 40$ to get that for $\dv \in \Omega(3)$,
                $$\sp_{a, \uh, b}(\dv) = \Sc(P, \sp)_{a, \uh, b}(\dv) =\Sc(\Pf, \Sf)_{a, \uh, b}(\dv) \rel{c + \mu_0}
                = \Sf_{a, \uh, b}(\dv) \rel{c}.$$

               {With $c$ taken sufficiently small,}  $t_0-1$ iterated applications of~\thref{c:contraction}  produce 
                $$
                 \cC^{t_0-1}(P, \sp)\proj (\dv) = \cC^{t_0-1}(\Pf, \Sf) \proj (\dv) \rel {e^{-t_0} }= \cC^{t_0-1}(\Pf, \Sf) \proj (\dv) \rel {\vartheta}
                $$
                for $\dv \in \Omega(3t_0)$, where we used $e^{-t_0} = (d_0n)^{-1} = O(\vartheta)$.
                Combining this bound with \eqref{eq:cinv} and \eqref{eq:cinv1}, we get 
                \bee        \label{eq:PSproof}
                    (\Pf, \Sf)\proj  (\dv)= (P, \sp)\proj (\dv) \rel{\vartheta } \text{ for }\dv \in \Omega(3t_0).
                \ee   {This equation is used below to establish~\thref{Pdense}.}  Finally, let $\dv  \in \D_0$, so that $\dv -\ea$ and $\dv -\eb$ are in $\Omega(3t_0)$ for any $a, b \in V$.
                Applying $\Rc$ once more, with Lemma~\ref{l:errorImplication}(a) and triangle inequality, gives
                 $$\Rf {\proj}(\dv) = \Rc(\Pf, \Sf) {\proj}(\dv)\rel{\zeta} = \Rc(P, \sp) {\proj}(\dv)\rel{(\mu_0 + \alpha)\vartheta} = R {\proj}(\dv)\rel{\zeta}.$$
                 
                 It remains to verify that for $q \geq 2$, $a \neq b \in V$ and $\dv \in {\D_0}$,
                 $$\Rf_{ab}(\dv) = \frac{d_a \left(\binom{n-1}{q} - d_b \right)}{d_b \left(\binom{n-1}{q} - d_a \right)}
        								\left(1 + \frac{(d_a-d_b)q}{d_0(1-\mu_0)(n-1-q)} \right) \rel{\zeta} ,$$
        								which is the expression claimed in~\eqref{eq:ratiostatement}. Indeed, introducing the parameters $\eps_a, \eps_b$ and expanding in $\alpha$, the right-hand can be rewritten as 
        			\begin{align*}
        			    &\frac{1+\eps_a}{1+\eps_b} \cdot \frac{1-\mu - \mu \eps_b}{1- \mu - \mu \eps_a} \left( 1+ \frac{\alpha(\eps_a-\eps_b)}{(1-\mu)(1-\alpha)}\right)\\
        			    =& \frac{1 + \eps_a - \mu - \mu \eps_a - \mu \eps_b}{1 + \eps_b - \mu - \mu \eps_a - \mu \eps_b} + O(\mu \eps^3) + \alpha \cdot \frac{(\eps_a-\eps_b)}{1-\mu}\rel{\eps}+ O(\alpha^2 \eps),
        			\end{align*}
                which coincides with the definition of $\Rf_{ab}(\dv)$ in~\eqref{eq:rf} and~\eqref{eq:rf3}, respectively, up to the error term $\zeta$. We remark that the ratio formula was stated in the form~\eqref{eq:ratiostatement} so that it can be readily used in the next step of the proof.
        	\end{proof}    
            
            \begin{remark}
        A hand-checkable proof for $q=2$ would be much simpler if one could exploit the symmetry in $\bade$ and $\badp$ in a rigorous way. This would entail showing that $\Pft_{a,\uh, b}(\dv')$ and $\Sft_{a,\uh, b}(\dv')$ remain symmetric in $\eps_a$ and $\eps_b$ up to $O(\eps^3)$-terms throughout the iterations of $\Cc$, and conceptually the simplest way to do that is to just compute the second-order terms in $\eps$. It would be interesting to find a more direct argument for this symmetry.
        \end{remark}
        Let us make it explicit how~\thref{Pdense} follows from the previous proof.
        \begin{proof}[Proof of~\thref{Pdense}]
                The hypothesis of the proposition is that the
                 parameters  $c, k, d, n, m, \mu, \varphi$ satisfy
                 $$k <c d^{\varphi}, \quad k \mu + k^2 / n < c \text { \quad and  \quad} \log^3 n < d.$$
                 These inequalities are assumed in~\thref{l:ratiosDense} with $d_0 = d$, $\eps = d^{-\varphi}$. Hence,~\eqref{eq:PSproof}, recalling the definition of $\Pf$ in~\eqref{eq:rf} and $\vartheta < k^2 d^{-2 \varphi}$  implies that for $\dv \in \D$,
                    \bee       \non
                        P_K(\dv) = \mu\left(1 + \frac{1 + \eps_K}{1-\alpha} \right) \rel{k^2d^{-2\varphi}}.
                     \ee
                Substituting $\mu = d \binom{n-1}{k-1}^{-1}$, $\alpha = \frac{k-1}{n-1}$ and 
                 $\eps_K= \sum_{v \in K} \frac{d_v-d}{d}$, we recover the expression claimed in~\thref{Pdense}. 
            \end{proof}

            		\begin{proof}[Proof of~\thref{t:denseCase}]

         	The relevant ratios $R_{ab}(\dv)$ were computed in \thref{l:ratiosDense}. The remainder of the proof is very similar to the sparse case, Theorem~\ref{t:sparseCase}. Some amount of repetition {is} necessary as certain conditions and error estimates need to be checked in either case. We emphasise that the sequences we are dealing with from now on have sum $dn$ or $dn+1$.  Recall that $\alpha = \frac{k-1}{n-1}$ and  $\sigma^2(\dv) = \frac 1n \sum_{v \in V}(d_v-d)^2$ for a sequence $\dv$ with average $d$. 
        	
         {Recall that the set $\D$ from the theorem statement contains the sequences of length $n$ with sum $ dn$ and spread at most $d^{1-\varphi}$.}             
            			\begin{claim}\lab{cl:Htilde}
            			   Let $H(\dv)=\pr_{\Bknm}(\dv) \widetilde H(\dv)$, where 
                $$ \widetilde H(\dv) = \exp\left( \frac{\alpha n}{2} - \frac{\alpha n \sigma^2(\dv)}{2d(1-\mu)(1-\alpha)} \right).$$ If $\dv-\ea$ and $\dv -\eb$ are elements of $\D$, then
                             \bel{eq:ratiosDense}
                             R_{ab}(\dv) = \frac{H(\dv-\ea)}{H(\dv-\eb)} \rel{\delta}
            			    \text{ \quad with \quad} \delta =
                    \begin{cases}
                        (\mu + \alpha)k^2 d^{-2\varphi}, & k \geq 4 \\
                        \mu d^{-3\varphi} + \alpha d^{-2\varphi}, & k = 3.
                        \end{cases}. \ee
            			\end{claim}
            		\begin{proof}
	We claim that if $\dv -\ea$ lies in $\D$ for $a \in V$, then $\dv$ satisfies the hypotheses of~\thref{l:ratiosDense} with $d_0 = d+1/n$ (the average of $\dv$) and $\eps = 2d^{- \varphi}$. Indeed, the spread of $\dv$ is at most $d^{1-\varphi}+1$ since the spread of $\dv-\ea$ is at most $d^{1-\varphi}$, and moreover, from the hypothesis $k^3d^{1-3\varphi}<c$ in~\thref{t:denseCase}, we deduce that ${\eps q<} 2k d^{-\varphi}<c^{1/3}$ for large $n.$ Hence we may apply~\thref{l:ratiosDense} and deduce that~\eqref{eq:ratiostatement} holds for $\dv$. 

            			To evaluate the right hand side of~\eqref{eq:ratiosDense} we use $km = dn$,~\eqref{eq:condbinom} and compute 
                \begin{align} \non
                    \frac{H(\dv-\ea)}{H(\dv-\eb)}&=   \frac{d_a \left(\binom{n-1}{k-1}- d_b +1\right)}{d_b \left( \binom{n-1}{k-1} -d_a +1 \right)}\exp\left( \frac{\alpha n}{2dn(1-\mu)(1-\alpha)} (2d_a - 2d_b) \right) \\
                \non
                    &= \frac{d_a \left(\binom{n-1}{k-1}- d_b \right)}{d_b \left( \binom{n-1}{k-1} -d_a \right)}\exp\left( \frac{\alpha (d_a - d_b) }{d(1-\mu)(1-\alpha)} \right) \relf{d^{1-\varphi}}{\binom{n-1}{k-1}^2}.
                \label{eq:ratioDense}
                \end{align}
                 whenever $\dv - \ea, \dv - \eb \in \D$.
                The estimate in the second line follows from~\eqref{eq:ratiocomp}. The expression on the right-hand side matches {the approximation~\eqref{eq:ratiostatement}} for the ratio $R_{ab}(\dv)$ up to a relative error $O(\eps^2 \alpha^2) = O(\delta) $.  Finally, to check that $d^{1-\varphi}\binom{n-1}{k-1}^{-2} = O(\delta) $, we split into two cases once again. For $k=3$, $d^{1-\varphi}\binom{n-1}{k-1}^{-2} = O\left(\alpha^4 d^{1-\varphi} \right)= O(\alpha d^{-2\varphi}) $. 
                For $k\geq 4$, we have $d^{1-\varphi}\binom{n-1}{k-1}^{-2} = O(\mu^2 d^{-1-\varphi}) = O(\delta)$.  We conclude that
                $\frac{ H(\dv-\ea)}{H (\dv- \eb) } = R_{ab}(\dv) \rel{\delta}$, as required.
            			\end{proof}

              If $d$ is an integer, let $\dv_{\reg}  = (d, d, \dots, d) \in \D$. For any sequence $\dv \in \D$, there is a sequence of sequences $\dv=\dv_1,\dv_2,\ldots, \dv_t=\dv_{\reg} \in \D$ with $t \leq 2nd^{1-\varphi}$ in which consecutive sequences are adjacent {(where we recall that two sequences are adjacent if they can be written as $\dv'-\ea$ and $\dv'-\eb$ for some $\dv' \in \Z^n$ and $a, b \in V$).} Let $r = nd^{1-\varphi}$. %
                By `telescoping'~\eqref{eq:ratiosDense}, we deduce that for $\dv \in \D$, 
              \bel{eq:telescopeD}
                \frac{\pr_{\Dknm}(\dv)}{H(\dv)} = c_1 {e^{O(r \delta)} },
              \ee
        where $c_1 =  \pr_{\Dknm}(\dv_{\reg})/H(\dv_{\reg})$.
         On the other hand, if $d$ is not an integer, then  we reach a similar conclusion by {defining  $\dv_{\reg}$ to be an appropriate sequence in} $\{ \lfloor d \rfloor, \lfloor d \rfloor +1\}^n$.
             
             We will show that $c_1 = 1 + {O(\eta_k)}$ using the fact that $\tilde{H}(\dv)$ is concentrated around 1 in both $\Bknm$ and $\Dknm$. {We isolate the following statement because it will be useful later, in proving \thref{t:denseCase}.}

        \begin{claim}\thlabel{cor:binapp_dense}
        Define
                    $$ \W = \left \{ \dv \in \D : \left| \sigma^2(\dv)-d(1- \mu) \right| \leq d \xi \right \},$$  %
        where  
        $$
     \xi = \frac{k \log^2 n}{\sqrt{n}}.
        $$
        Then $\pr_{\Dknm}(\W)= 1 - O \left(n^{-\omega(1)}\right)$  and $\pr_{\Bknm}(\W) = 1 - O \left(n^{-\omega(1)}\right)$.
        Moreover, uniformly for $\dv  \in \W$,  
        \bel{newD}
        \pr_{\Dknm}(\dv) = \pr_{\Bknm}(\dv) (1+ {O(\eta_k)}) . %
        \ee
                \end{claim}
                \begin{proof}
                    We start by estimating the probability of $\W$ in the two probability spaces. \thref{l:sigmaConc}(i)  implies that if $\dv$ is sampled from $\Bknm$ or $\Dknm$, $|d_v -d| \leq 20 \sqrt{d \log n} < d^{1-\varphi}$ for all $v \in V$ with probability $1- O \left(n^{-\omega(1)}\right)$ (for the last inequality, we used the assumption {of~\thref{t:denseCase}} that $\log^C n = o(d)$ for any integer $C$).  {Thus   $\dv\in\D$.}  
                     
                     Secondly, we claim that if $\dv$ is sampled from $\Bknm$ or $\Dknm$ then         $\sigma^2 (\dv) = d(1-\mu)(1\pm \xi )$ with probability $1-O \left(n^{-\omega(1)}\right)$.
                    Indeed, this follows from~\thref{l:sigmaConc}(ii)  with $\beta = \log n/4$, by noticing that in this case {the second term in the deviation is    {dominated by $\xi^2=O(\xi)$} since $k^3 < cd^{{3 \varphi-1}} < cd^{1/2}$}, and hence  $\frac{ k \log^2 n}{d} <  d^{-1/2} \log^2 n \to 0$. By definition of $\W$, we may now conclude that
                    $$\pr_{\Bknm}(\W) = 1 - O \left(n^{-\omega(1)}\right) \quad \text{and} \quad \pr_{\Dknm}(\W) = 1 - O \left(n^{-\omega(1)}\right).$$

                    {It is easy to check that  $\tilde{H}(\dv) = e^{O(k \xi)}$}   for $\dv \in \W$, {and hence} 
                    \bee
                          \sum_{\dv \in \W} H(\dv) = { e^{O(k \xi)}}\sum_{\dv \in \W} \pr_{\Bknm}(\dv)   
                         =  { e^{O(k \xi)}}.
                     \ee
                    Using~\eqref{eq:telescopeD} and the above estimates, we get  
                    $$
                   \pr_{\Dknm}(\W) =\sum_{\dv \in \W} \pr_{\Dknm}(\dv) =  c_1 {e^{O(r \delta)}} \sum_{\dv \in \W} H(\dv)  = c_1 {e^{O(r \delta + k \xi)}}.
                    $$
                           From this estimate and $\pr_{\Dknm}(\W) =  1 -O \left(n^{-\omega(1)}\right) $, we deduce that {$c_1  = e^{O \left( r \delta + k \xi + n^{-\omega(1)} \right)} = e^{O(\eta_k)}, $}
                           using $r = nd^{1-\varphi}$ and the definition of $\eta_k$ from the theorem statement {(and noting that $\mu=O(d/n^2)$ when $k=3$).}
                   Hence, {recalling the definition  $H(\dv)=\pr_{\Bknm}(\dv) \widetilde H(\dv)$ in Claim~\ref{cl:Htilde},}~\eqn{eq:telescopeD} implies that for $\dv \in \D$,
                    \bee
                        \label{eq:telescope1D}
                            \pr_{\Dknm}(\dv) = \pr_{ \Bknm } (\dv) \tilde{H}(\dv) {e^{O(\eta_k)}}.
                    \ee
                    In particular, for $\dv \in \W \subset \D$, {since $ \tilde{H}(\dv) =  e^{ O(k\xi)}= e^{O(\eta_k)}$, we have~\eqn{newD} provided that $\eta_k$ is bounded}. 
                    {So, to finish the proof of the claim, it only remains to check the last assertion in the theorem,} that $\eta_k < 3c$ for $n$ sufficiently large. Large $n$ is assumed implicitly in the rest of the proof.
                
                Firstly, let $k \geq 4$. The assumptions of the theorem include    $k^2\log^2 n/\sqrt{n} < c$ and $k^3 d^{1 - 3\varphi} < c$.
                Recall that $\frac{4}{9} <\varphi< \frac 12$.  For the remaining term in $\eta_k$, we have  $d = \mu \binom{n-1}{k-1}  \geq \mu \left( \frac nk \right)^{k-1}$, so  
                \begin{align*}
                    \mu k^2 n d^{1- 3 \varphi} \leq \mu^{2- 3\varphi} k^2 n \left(\frac kn \right)^{(3\varphi - 1) (k-1)} \leq   k^2 n \left(\frac kn \right)^{(3\varphi - 1)(k-1)} 
                \end{align*}
 If $k=O(1)$ this is  $O(n^{1-(3\varphi -1)(k-1)})=o(1)$,  as $3\varphi -1 > \frac{1}{3}$ and $k-1 \geq 3$. On the other hand,  since $k<n^{1/4}$,  the right-hand side is at most $n^{3/2- (k-1)/4}=o(1)$ when  $k\ge 8$. So   indeed $\eta_k <3c$.
      
                For $k = 3$, we have $ d^{1-3\varphi} <c/9$ by {the theorem's} assumption. Moreover, {$d\le \binom{n}{2}$ and so} {$d^{2-4\varphi}/n \le n^{4-8\varphi -1} < n^{-5/9}$} and hence $\eta_3 < 3c$.                     This completes the proof of the claim. 
                 \end{proof}
                    The theorem follows from~\eqn{eq:telescope1D} considering the definition of $\tilde{H}(\dv)$.  
        \end{proof}
      	\begin{proof} [Proof of \thref{t:main}] We use
      	\thref{cor:binapp_dense,cor:binapp_sparse}. 
      	Fix $k = k(n) \leq n^{ C}$. We distinguish two cases for $d :=km/n$. Let $\eps=  \frac19-C>0 $.  If  $k^6< d^{1-\eps}$  and $d > n^{1/3}$, %
        		apply Claim~\ref{cor:binapp_dense} 
        		with  $\varphi =  \frac12 - \frac{\eps}{12}$, and with the parameters restricted as in \thref{t:denseCase}. Note that in the case $k=3$ we have $d^{2-4\phi}<d^{1/27}<n^{1/9}$. Also note that for general $k$, we have   $k^3 d^{1- 3\varphi} < d^{ -\eps/4 }  = o(1)$. Moreover, for $k\ge 4$ we have $\mu n k^2 d^{1-3\phi}<\mu n k^2/d^{4/9} = d^{5/9}n^{19/9}/\binom{n}{k} = o(1)$ since $d<n^{k-1}$.
Hence,  $\eta_k = o(1)$  and the assumptions of the claim are satisfied. The set $\W$ has the desired properties. 
        		
        		Secondly, let  $d \leq \max \{  k^{6/(1-\eps)}, n^{1/3} \} < n^{2/3}$ (since $k \le n^{1/9-\eps}$).   Recall that for Claim~\ref{cor:binapp_sparse}, the parameter restrictions and definitions, in particular that of $\psi_k$, are inherited from \thref{t:sparseCase}. Defining  $\Delta_* =  \lceil d \rceil$, we will show that $\pseta = o(1)$. Indeed, for $k \geq 4$, using $m=dn/k$ we have 
\bean
\psi_k &\leq& \frac{  k^2}{m} \left((d+1)^2 + k(d+1) \log^4 n +k^2 \log^9 n \right) \\
       &=&  \frac{ k^3 d + O(k^4 \log^4 n)}{n} +\frac{ k^4 \log^9n}{m}.\eean
This bound is easily seen to be $o(1)$ because $k^3d<n^{1/3-3\eps}n^{2/3}$, and using the assumption  of \thref{t:main} that $m =  dn/k > \sqrt{n}$. 
 		       		
        		In  the  case $k=3$, we necessarily have $d \leq n^{1/3}$, and thus $\psi_3 =O\big(m^{-1}(d^3 + \log^9 n)\big)= O(n^{-1/3})$.
        		Applying Claim~\ref{cor:binapp_sparse} gives the desired conclusion  in either case.
        		\end{proof}

%




\bibliographystyle{amsplain}


\begin{aicauthors}
\begin{authorinfo}[nina]
Nina Kam\v cev \\
Department of Mathematics \\
Faculty of Science\\
University of Zagreb\\
Croatia\\
nina\imagedot{}kamcev\imageat{}math\imagedot{}hr
\end{authorinfo}
\begin{authorinfo}[anita]
  Anita Liebenau\\
 School of Mathematics and Statistics\\
 UNSW Sydney NSW 2052\\
 Australia\\
 a\imagedot{}liebenau\imageat{}unsw\imagedot{}edu\imagedot{}au 
\end{authorinfo}
\begin{authorinfo}[nick]
  Nick Wormald\\
  School of Mathematics\\
  Monash University VIC 3800\\
  Australia \\
  nick\imagedot{}wormald\imageat{}monash\imagedot{}edu
\end{authorinfo}
\end{aicauthors}

\end{document}